\documentclass[reqno]{amsart}
\usepackage{amsmath,amsfonts,amssymb}
\usepackage{hyperref}

\usepackage{verbatim}
\usepackage{xy}
\input xy
\xyoption{all}
\usepackage[shortlabels]{enumitem}
\usepackage{tikz-cd}
\usepackage{url} 
\usepackage[margin=3.35cm]{geometry} %allows customization of margins and changing paper sizes
\usepackage[toc,page]{appendix} %Adding appendix
\usepackage{dsfont}

\newtheorem{satz}{Satz}[section]
\newtheorem{Theorem}[satz]{Theorem}
\newtheorem{Cor}[satz]{Corollary}
\newtheorem{Lemma}[satz]{Lemma}
\newtheorem{Prop}[satz]{Proposition}

\theoremstyle{definition}

\newtheorem{Remark}[satz]{Remark}
\newtheorem{Example}[satz]{Example}
\newtheorem*{Def*}{Definition}
\newtheorem*{Example*}{Example}
\newtheorem*{Theorem*}{Theorem}
\newtheorem*{Prop*}{Proposition}
\newtheorem*{Remark*}{Remark}

\newcommand{\R}{\mathbb{R}}

\newcommand{\Z}{\mathbb{Z}}
\newcommand{\C}{\mathbb{C}}

\newcommand{\N}{\mathbb{N}}

\newcommand{\HH}{\mathbb{H}}

\newcommand{\xr}{\xrightarrow}

\newcommand{\Ho}{\operatorname{Ho}}

\newcommand{\CP}{\mathbb{C}\text{P}}
\newcommand{\RP}{\mathbb{R}\text{P}}
\newcommand{\co}{\colon\thinspace}
\newcommand{\Star}{\mbox{\smaller[3]$\bigstar$}}

\DeclareMathOperator{\Rep}{Rep}

\DeclareMathOperator{\hocolim}{hocolim}

\DeclareMathOperator{\FI}{FI}
\DeclareMathOperator{\Ind}{Ind}
\DeclareMathOperator{\id}{id}

\DeclareMathOperator{\Comm}{Comm}
\DeclareMathOperator{\Sym}{Sym}

\begin{document}

\begin{abstract}
We describe the~$C_2$--equivariant homotopy type of the space of commuting~$n$--tuples in the stable unitary group in terms of Real K-theory. The result is used to give a complete calculation of the homotopy groups of the space of commuting~$n$--tuples in the stable orthogonal group, as well as of the coefficient ring for commutative orthogonal K--theory.
\end{abstract}

\title[Commuting matrices and Atiyah's real K-theory]{Commuting matrices and Atiyah's real K-theory}

\author{Simon Gritschacher}
\address{University of Copenhagen\\ Copenhagen, Denmark} 
\email{gritschacher@math.ku.dk}

\author{Markus Hausmann}
\address{University of Copenhagen\\ Copenhagen, Denmark} 
\email{hausmann@math.ku.dk}

\maketitle

\section{Introduction}
For a topological group~$G$, the space of commuting~$n$--tuples of elements in~$G$ is defined as
\[
C_n(G)=\{(g_1,\dots,g_n)\in G^{\times n}\,|\, g_i g_j=g_j g_i \text{ for all }i,j\} ,
\]
viewed as a subspace of~$G^{\times n}$, or, in other words, the space of group homomorphisms~$\Z^n\to G$.
Closely related is the representation space~$\Rep_n(G)$, which is defined as the quotient of~$C_n(G)$ by the action of~$G$ by conjugation.
These spaces are of classical interest because of their relevance in mathematical physics \cite{Wi82, KS00},
and also have interesting homotopy-theoretic properties: For example, they admit stable splittings \cite{AC07,BJS11}, they give rise to the spectrum of commutative K-theory \cite{ACT12,AG15,AGLT17}, and they satisfy rational homological stability when~$G$ runs through any of the sequences of classical Lie groups \cite{RS18}. We refer to \cite{CS16a} for a survey.

The purpose of the present article is to describe the homotopy type of the spaces of commuting elements in the stable unitary group~$U$ and the stable orthogonal group~$O$, and to use this description to compute their homotopy groups. Instead of considering the unitary case and the orthogonal case separately, it turns out to be useful to package them into a single object with~$C_2$--action and study this object through equivariant homotopy theory. Accordingly, we write~$C_n^\R$ for the space~$C_n(U)$ viewed as a~$C_2$--space via complex conjugation, with fixed point space~$C_n(O)$. Likewise, we write~$\Rep_n^{\R}$ for~$\Rep_n(U)$ viewed as a~$C_2$--space, with fixed point space~$\Rep_n(O)$. Our main result is a description of the equivariant homotopy type of~$C_n^{\R}$ and~$\Rep_n^{\R}$. 

Our description is in terms of the connective Real K-theory spectrum~$k\R$ introduced by Atiyah \cite{Ati66}. This is a genuine~$C_2$--spectrum whose underlying spectrum is equivalent to connective unitary K-theory~$ku$ and whose categorical fixed point spectrum~$k\R^{C_2}$ is equivalent to connective orthogonal K-theory~$ko$. The~$0$--th Postnikov section of~$k\R$ is the Eilenberg-MacLane spectrum for the constant Mackey functor~$\underline{\Z}$, which we denote by~$H\underline{\Z}$. Finally, we let~$S^{\sigma}$ denote the unit sphere in~$\C$, based at~$1$ and regarded as a~$C_2$--space via complex conjugation.
Then we show:
\begin{Theorem} \label{thm:main}
There are~$C_2$--equivalences
\[
C_n^{\R}\simeq \Omega^{\infty}(k\R\wedge (S^{\sigma})^{\times n})\,, \hspace{0.8cm} \Rep_n^{\R}\simeq \Omega^{\infty}(H\underline{\Z}\wedge  (S^{\sigma})^{\times n}).
\]
Under these equivalences, the map~$C_n^{\R}\to \Rep_n^{\R}$ is induced by the~$0$--th Postnikov section~$k\R\to H\underline{\Z}$.
\end{Theorem}
The proof is geometric and based on the~$C_2$--equivariant version of Segal's model for connective K-theory \cite{Seg77}. Here,~`$C_2$--equivalence' means that there is a zig-zag of~$C_2$--equivariant maps inducing weak equivalences both on underlying spaces and on~$C_2$--fixed points. Hence, on underlying spaces the theorem gives equivalences
\[
C_n(U)\simeq  \Omega^{\infty}(ku\wedge (S^1)^{\times n})\,, \hspace{0.8cm} \Rep_n(U)\simeq \Omega^{\infty}(H\Z\wedge (S^1)^{\times n})\, .
\]
These equivalences can also be obtained by combining two results of Lawson on unitary deformation K-theory, the product formula \cite{Law06} and the Bott exact sequence \cite{Law09}. Our focus is on the equivalences on fixed points
\[
C_n(O)\simeq  \Omega^{\infty}(k\R\wedge (S^{\sigma})^{\times n})^{C_2}\,, \hspace{0.8cm} \Rep_n(O)\simeq \Omega^{\infty}(H\underline{\Z}\wedge (S^\sigma)^{\times n})^{C_2}\, ,
\]
describing the spaces of commuting orthogonal matrices for which no analogues of Lawson's results are known. Since~$(S^{\sigma})^{\times n}$ splits stably as a wedge over the representation spheres~$S^{k\cdot \sigma}$, each repeated ${n}\choose{k}$ times, the spaces~$C_n(O)$ and~$\Rep_n(O)$ decompose accordingly as products over spaces of the form~$\Omega^{\infty}(k\R\wedge S^{k\cdot \sigma})^{C_2}$ respectively~$\Omega^{\infty}(H\underline{\Z}\wedge S^{k\cdot \sigma})^{C_2}$. The homotopy groups of these infinite loop spaces are so-called~$RO(C_2)$--graded homotopy groups of ~$k\R$ and~$H\underline{\Z}$, which are known due to work of Dugger \cite{Du05} and Bruner-Greenlees \cite{BG10}. Hence, we obtain a complete computation of~$\pi_*(C_n(O))$ and~$\pi_*(\Rep_n(O))$, see Sections \ref{sec:cno} and \ref{sec:rno}.
\begin{Remark}
The~$C_2$--spaces~$\Omega^{\infty}(k\R\wedge S^{k\cdot \sigma})$ and~$\Omega^{\infty}(H\underline{\Z}\wedge S^{k\cdot \sigma})$ appearing here are~$(k\cdot \sigma)$--fold equivariant classifying spaces of the~$C_2$--infinite loop spaces~$\Z\times BU$ and~$\Z$. In particular, they satisfy
\[
\Omega^{k\cdot \sigma}(\Omega^{\infty}(k\R\wedge S^{k\cdot \sigma}))\simeq \Z\times BU\,, \hspace{0.8cm} \Omega^{k\cdot \sigma}(\Omega^{\infty}(H\underline{\Z}\wedge S^{k\cdot \sigma}))\simeq \Z\,.
\]
Since the~$C_2$--action on~$S^{k\cdot \sigma}$ is non-trivial, the fixed points of these deloopings -- which describe the factors of~$C_n(O)$ and~$\Rep_n(O)$ -- are  not themselves iterated classifying spaces of the fixed points~$\Z\times BO$ and~$\Z$. In general we know of no simple way to describe the homotopy type of the spaces~$\Omega^{\infty}(k\R\wedge S^{k\cdot \sigma})^{C_2}$ -- and hence of~$C_n(O)$ -- without passing through equivariant homotopy theory. On the other hand, the spaces~$\Omega^{\infty}(H\underline{\Z}\wedge S^{k\cdot \sigma})^{C_2}$ are products of Eilenberg-MacLane spaces, since the fixed point spectra~$(H\underline{\Z}\wedge S^{k\cdot \sigma})^{C_2}$ are modules over~$(H\underline{\Z})^{C_2}\simeq H\Z$. Hence, their homotopy type is determined by their homotopy groups, which are well-known.
\end{Remark}

We also describe a~$C_2$--equivariant version of commutative K--theory. In \cite{ACT12} Adem, Cohen and Torres-Giese observed that the spaces~$C_n(G)$ for varying~$n$ can be assembled into a simplicial space, denoted~$C_\bullet(G)$, by restricting the simplicial structure of the bar construction~$B_\bullet G$. The geometric realization of~$C_\bullet (G)$ is denoted~$B_{\textnormal{com}}G$. Commutative K--theory is the cohomology theory represented by~$B_{\textnormal{com}}G$, when~$G$ is any of the classical infinite matrix groups. The theory was introduced by Adem and G\'omez in \cite{AG15}, and further studied by Adem, G\'omez, Lind and Tillmann in \cite{AGLT17} and the first author in \cite{Gri18}. Let ~$B_{\textnormal{com}}^{\R}$ denote the geometric realization of the simplicial~$C_2$--space~$C_{\bullet}^{\R}$. Using Theorem \ref{thm:main} we prove:
\begin{Theorem} \label{thm:bcom} There is a~$C_2$--equivalence
\[ B_{\textnormal{com}}^{\R}\simeq \Omega^{\infty}(k\R\wedge \CP^{\infty})\, , \]
where $C_2$ acts on $\CP^{\infty}$ through complex conjugation.
\end{Theorem}
From this equivalence we deduce the homotopy groups of~$B_{\textnormal{com}}O$. Indeed, the theorem implies equivalences
\[ B_{\textnormal{com}}U\simeq \Omega^{\infty}(ku\wedge \CP^{\infty})\,, \hspace{0.8cm} B_{\textnormal{com}}O\simeq \Omega^{\infty}(k\R\wedge \CP^{\infty})^{C_2}.\]
The former was the main result of \cite{Gri18}. Since the~$C_2$--spectrum~$k\R\wedge \CP^{\infty}$ splits into the wedge over all~$k\R \wedge S^{n\cdot (1+\sigma)}$ for~$n>0$, the homotopy groups of~$B_{\textnormal{com}}O$ can again be described as a sum of~$RO(C_2)$--graded homotopy groups of~$k\R$. The splitting is not multiplicative and it is more involved to describe~$\pi_*(B_{\textnormal{com}}O)$ as a non-unital ring, which we work out in Section \ref{sec:bcomo}.

\subsection*{Acknowledgements} We thank Stefan Schwede and the anonymous referee for helpful comments. Both authors were supported by the Danish National Research Foundation through the Centre for Symmetry and Deformation (DNRF92).

\section{Proofs of the main results} \label{sec:main}
In this section we give the proofs of Theorems \ref{thm:main} and \ref{thm:bcom}. In Section \ref{sec:homgroups} we apply these theorems to compute the homotopy groups of various spaces of commuting elements and spaces of representations.

To obtain a geometric model for the~$C_2$--space~$\Omega^{\infty}(k\R\wedge (S^{\sigma})^{\times n})$ we make use of the theory of~$\Gamma$--spaces, as developed by Segal \cite{Seg74} and Bousfield-Friedlander \cite{BF78} non-equivariantly, and extended to the equivariant context by Shimakawa \cite{Shi89}, Ostermayr \cite{Ost16}, May-Merling-Osorno \cite{MMO17}, Schwede \cite[Appendix B]{Sch18} and others. We briefly recall how equivariant~$\Gamma$--spaces give rise to equivariant spectra. For a more detailed treatment we refer to the sources above.

Let~$G$ be a finite group. A~$G-\Gamma$--space is a functor~$X\co Fin_*\to (G-spaces)_*$ from finite pointed sets to based~$G$--spaces satisfying~$X(\{*\})=\{*\}$.
Every~$G-\Gamma$--space~$X$ can be extended to a functor
\[ X(-)\co (G-spaces)_*\to (G-spaces)_* \]
via the coend formula
\begin{equation} \label{eq:coend}
X(A)=\int_{k_+} X(k_+)\times A^{\times k}\,.
\end{equation}
We make this construction more explicit for Real K-theory below. The~$G$--action on~$X(A)$ is induced by the diagonal action on~$X(k_+)\times A^{\times k}$. Given another based~$G$--space~$B$, there is an assembly map~$X(A)\wedge B\to X(A\wedge B)$ induced by the diagonal maps~$B\to B^{\times k}$. In particular, every~$G-\Gamma$--space~$X$ gives rise to a genuine~orthogonal~$G$--spectrum~$X(\mathbb{S})$ in the sense of Mandell-May \cite{MM02} by setting~$X(\mathbb{S})(V)=X(S^V)$ for an orthogonal~$G$--representation~$V$, with structure maps
\[ X(\mathbb{S})(V)\wedge S^W=X(S^V)\wedge S^W\to X(S^V\wedge S^W)\cong X(S^{V\oplus W})=X(\mathbb{S})(V\oplus W) \]
given by the assembly maps above. If the~$G-\Gamma$--space~$X$ is~$G$--cofibrant and special (see \cite[Definitions B.33 and B.49]{Sch18} for a definition of these notions), the~$G$--spectrum~$X(\mathbb{S})$ is a positive~$G-\Omega$--spectrum, meaning that the adjoint structure maps
\begin{equation*} X(V)\to \Omega^W (X(V\oplus W)) \end{equation*}
are~$G$--equivalences whenever~$V^G\neq 0$. One consequence is that, if~$X$ is~$G$--cofibrant and special, the infinite loop space~$\Omega^{\infty}X(\mathbb{S})$ (which is computed as the~$0$--th term of a~$G-\Omega$--spectrum replacement of~$X(\mathbb{S})$) is~$G$--equivalent to~$\Omega X(S^1)$. More generally, the following holds:
\begin{Prop} \label{prop:special} Let~$X$ be a~$G$--cofibrant and special~$G-\Gamma$--space and~$A$ a based~$G$--CW complex. Then~$\Omega^{\infty}(X(\mathbb{S})\wedge A)$ is~$G$--equivalent to~$\Omega X(S^1\wedge A)$.
\end{Prop}
\begin{proof} We first assume that~$A$ has finitely many cells and let~$X_A$ denote the prolonged~$G-\Gamma$--space sending a finite set~$k_+$ to the evaluation~$X(k_+\wedge A)$.
Then the assembly maps induce a map of orthogonal~$G$--spectra
\[ X(\mathbb{S})\wedge A\to X_A(\mathbb{S})\,, \]
which is a~$G$--stable equivalence by \cite[Proposition 3.6]{Blu06}. In particular, there is a~$G$--equivalence $\Omega^{\infty}(X(\mathbb{S})\wedge A)\simeq \Omega^{\infty}(X_A(\mathbb{S}))$. Now, if $A$ is $G$--homeomorphic to the geometric realization of a finite based $G$--simplicial set, then by \cite[Propositions B.37(ii) and B.54(iii)]{Sch18}, the prolonged~$G-\Gamma$--space~$X_A$ is again~$G$--cofibrant and special. Thus in that case we have
\[ \Omega X(S^1\wedge A)\cong \Omega X_A(S^1)\simeq \Omega^{\infty}X_A(\mathbb{S})\simeq \Omega^{\infty}(X(\mathbb{S})\wedge A)\,. \]
Since every finite based $G$--CW complex is $G$--homotopy equivalent to the geometric realization of a finite based $G$--simplicial set (see \cite[Proposition B.46(ii)]{Sch18}) and all terms in the above chain of equivalences are homotopy-invariant, the statement follows for all finite $A$.

To reduce the infinite case to the finite one, we write an arbitrary $G$--CW complex as the colimit over its finite $G$--subcomplexes and use that $X(-)$ commutes with filtered colimits along closed inclusions (the fact that $X(-)$ takes closed inclusions to closed inclusions is \cite[Proposition B.26 (ii)]{Sch18}). The poset of finite $G$--subcomplexes of $A$ satisfies the conditions of \cite[Proposition A.15]{Sch18}, and hence every map from a compact space to the colimit $X(A)$ factors through $X(A')$ for some finite $A'\subset A$. Applying this more generally to the evaluations~$X(S^V\wedge A)$,
one derives that also for infinite~$A$ the assembly map $X(\mathbb{S})\wedge A\to X_A(\mathbb{S})$ is a $G$--stable equivalence and the~$G$--spectrum~$X_A(\mathbb{S})$ is a positive~$G-\Omega$--spectrum. The statement follows.
\end{proof}
Another consequence of~$X$ being~$G$--cofibrant and special is that the natural map~$X(A\vee A)\to X(A)\times X(A)$ is a~$G$--equivalence for all based~$G$--CW-complexes~$A$ (see \cite[Proposition B.54]{Sch18} for finite~$A$, the infinite case again follows from both sides commuting with filtered colimits along closed inclusions, cf. the proof of Proposition \ref{prop:special} above). In particular, for every subgroup~$H\leq G$ this defines an abelian monoid structure on the component sets~$\pi_0(X(A)^H)$ via the composite
\[ \pi_0(X(A)^H)\times \pi_0(X(A)^H)\xleftarrow{\cong} \pi_0(X(A\vee A)^H) \xrightarrow{\text{fold}_*} \pi_0(X(A)^H)\,. \]
The following holds:
\begin{Prop} \label{prop:veryspecial} Let~$X$ be a~$G$--cofibrant and special~$G-\Gamma$--space and~$A$ a based~$G$--CW complex such that~$\pi_0(X(A)^H)$ is an abelian group for all subgroups~$H$ of~$G$. Then the map~$X(A)\to \Omega X(S^1\wedge A)$ is a~$G$--equivalence and hence~$\Omega^{\infty}(X(\mathbb{S})\wedge A)$ is~$G$--equivalent to~$X(A)$.
\end{Prop}
\begin{proof} Again we can reduce to finite~$A$, where the statement is \cite[Theorem B.61]{Sch18}.
\end{proof}
The condition that the~$\pi_0(X(A)^H)$ are groups is automatic if all the fixed point spaces~$A^H$ are path-connected, since this implies that all~$X(A)^H$ are also path-connected. One way to see this is to use that such an $A$ is $G$-homotopy equivalent to the geometric realization of a $G$-simplicial set $A_{\bullet}'$ with a single $0$-simplex, and hence $X(A)$ is $G$-homotopy equivalent to the realization of the $G$-simplicial space $X(A_{\bullet}')$ (this makes use of the fact that evaluation $X(-)$ commutes with geometric realizations, cf. \cite[Proposition B.29]{Sch18}). This $G$-simplicial space again has a single $0$-simplex, and it follows that all the fixed points of its realization are path-connected, as desired.

Finally we note that if a~$G-\Gamma$--space~$X$ is~$G$--cofibrant and special, then all its fixed point~$\Gamma$--spaces~$X^H$ for~$H\subset G$ are cofibrant and special in the non-equivariant sense, and the realization~$X^H(\mathbb{S})$ is stably equivalent to the derived categorical fixed point spectrum~$X(\mathbb{S})^H$. 

Now we turn to the case~$G=C_2$ and Real K-theory. Let~$k\R$ denote the~$C_2-\Gamma$--space which assigns to a finite pointed set~$k_+$ the space of~$k$-tuples~$(V_1,\hdots,V_k)$ of pairwise orthogonal finite dimensional complex subspaces of~$\C^{\infty}$, with~$C_2$--action by complex conjugation on~$\C^{\infty}$. Given a map of finite pointed sets~$\alpha\co k_+\to l_+$, the induced map~$k\R(\alpha)$ sends a $k$-tuple~$(V_1,\hdots,V_k)$ to the $l$--tuple~$(\bigoplus_{i\in \alpha^{-1}(1)} V_i,\hdots,\bigoplus_{i\in \alpha^{-1}(l)} V_i)$. To make the topology on~$k\R(k_+)$ precise, we write 
\[ k\R(k_+)=\bigsqcup_{n_1,\hdots,n_k\in \N} \left(L(\C^{n_1}\oplus \cdots \oplus \C^{n_k},\C^{\infty})/\left(U(n_1)\times\hdots\times U(n_k)\right)\right). \]
Here,~$L(V,W)$ denotes the space of complex linear isometric embeddings between two complex inner product spaces, of which we assume~$V$ to be finite dimensional. If~$W$ is also finite dimensional and without loss of generality~$\dim W\geq \dim V$, we choose a linear isometric embedding~$f\co V\hookrightarrow W$ to identify~$L(V,W)$ with~$U(W)/U(\text{im}(f)^{\perp}))$ and use this to define the topology on~$L(V,W)$. One can show that this topology does not depend on the chosen~$f$. For infinite~$W$, we give~$L(V,W)$ the colimit topology over all~$L(V,W')$ with~$W'\subset W$ a finite dimensional subspace.

Note that the space~$k\R(1_+)$ is the disjoint union over all Grassmanians~$\text{Gr}_n$ of~$n$--dimensional complex subspaces with conjugation action, which is a classifying space for~$n$--dimensional Real vector bundles in the sense of Atiyah \cite{Ati66}. The underlying~$\Gamma$--space of~$k\R$ is precisely Segal's model for connective complex K-theory \cite{Seg77}. Moreover, a configuration~$(V_1,\hdots,V_k)$ is fixed under the conjugation action if and only if each~$V_i$ is the complexification of a real subspace~$V_i'$ of~$\R^{\infty}$. Hence, one finds that the fixed points~$\Gamma$--space~$k\R^{C_2}$ is isomorphic to the~$\Gamma$--space of tuples of pairwise orthogonal finite dimensional real subspaces of~$\R^{\infty}$, i.e., Segal's model for connective orthogonal K-theory.

The realization~$k\R(\mathbb{S})$ which, by abuse of notation, we denote again by~$k\R$, is a model for connective Real K-theory.  Since~$k\R$ is a~$C_2$--cofibrant special~$C_2-\Gamma$--space (see below), the previous discussion shows that the underlying spectrum of~$k\R$ is connective unitary K-theory~$ku$, and the derived categorical fixed point spectrum~$k\R^{C_2}$ is given by connective orthogonal K-theory~$ko$.

\begin{Lemma} \label{lem:special} The~$C_2-\Gamma$--space~$k\R$ is~$C_2$--cofibrant and special.
\end{Lemma}
\begin{proof} To show that~$k\R$ is special, one needs to consider the maps
\[ P_k\co  k\R(k_+)\xrightarrow{\prod p_i} k\R(1_+)^{\times k} \]
induced by the projections~$p_i\co k_+\to 1_+$ sending~$i$ to~$1$ and all other points to the basepoint (see \cite[Proposition B.52]{Sch18}). Concretely, the map~$P_k$ sends a~$k$-tuple of finite dimensional pairwise orthogonal subspaces of~$\C^{\infty}$ to the same tuple, forgetting the fact that the subspaces were mutually orthogonal. This map~$P_k$ is~$(C_2\times \Sigma_k)$--equivariant. The spectrum~$k\R$ being special is equivalent to this map being a weak equivalence on underlying spaces and on all fixed points of graph subgroups~$\{(x,\alpha(x))\ |\ x\in C_2\}$ for group homomorphisms~$\alpha:C_2\to \Sigma_k$, for all~$k\geq 2$. Choosing a~$C_2$--equivariant unitary embedding~$\phi\co \C^k\otimes_{\C}\C^{\infty}\hookrightarrow \C^{\infty}$ induces a~$(C_2\times \Sigma_k)$--equivariant map in the other direction by sending a tuple~$(V_1,\hdots,V_k)$ to the pairwise orthogonal tuple~$(\phi(e_1\otimes V_1),\hdots,\phi(e_k\otimes V_k))$, where~$e_i$ is the standard~$i$--th basis vector of~$\C^k$. Here,~$C_2$ acts on~$\C^k\otimes_{\C}\C^{\infty}$ by complex conjugation on~$\C^{\infty}$ and trivially on~$\C^k$. After taking fixed points for a graph subgroup, these two maps become inverse homotopy equivalences. The proof for this claim is completely analogous to the one for the~$G-\Gamma$--space~$ku_G$ modeling connective equivariant~K-theory, which is discussed in detail in \cite[Theorem 6.3.19]{Sch18}, so we refrain from giving it here. It reduces to the fact that the spaces~$L(V,\C^{\infty})$ are~$(C_2\ltimes U(V))$--cofibrant with contractible fixed points~$L(V,\C^{\infty})^H$ for all subgroups~$H\subset C_2\ltimes U(V)$ having trivial intersection with~$1\times U(V)$. Likewise, the proof that~$k\R$ is~$C_2$--cofibrant is analogous to \cite[Example 6.3.16]{Sch18}.
\end{proof}
Let~$A$ be a based~$C_2$--space. By definition, a point in~$k\R(A)$ is represented by an element of~$k\R(k_+)\times A^{\times k}$, i.e, by a tuple~$(V_1,\hdots,V_k;x_1,\hdots,x_k)$ where the~$x_i$ are points in~$A$ and the~$V_i$ are pairwise orthogonal finite dimensional complex subspaces of~$\mathbb{C}^{\infty}$. The non-trivial element~$\tau\in C_2$ acts via 
\[ \tau\cdot (V_1,\hdots,V_k;x_1,\hdots,x_k)=(\overline{V_1},\hdots,\overline{V_k};\tau\cdot x_1,\hdots,\tau\cdot x_k). \]
One can think of~$(V_1,\hdots,V_k;x_1,\hdots,x_k)$ as a labeled configuration of points in~$A$, with~$V_i$ being the label for~$x_i$. The equivalence relation on these points encoded in the coend formula (\ref{eq:coend}) is generated by
\begin{itemize}
\item~$(V_1,\hdots,V_k;x_1,\hdots,x_k)\sim (V_{\sigma(1)},\hdots,V_{\sigma(k)};x_{\sigma(1)},\hdots,x_{\sigma(k)})$ for every permutation~$\sigma\in \Sigma_k$,
\item~$(V_1,\hdots,V_k;x_1,\hdots ,x_{k-1},*)\sim (V_1,\hdots,V_{k-1};x_1,\hdots,x_{k-1})$ for~$*$ the basepoint in~$A$,
\item~$(V_1,\hdots,V_{k-1},0;x_1,\hdots,x_k)\sim (V_1,\hdots,V_{k-1};x_1,\hdots,x_{k-1})$, and
\item~$(V_1,\hdots,V_k;x_1,\hdots,x_k)\sim (V_1,\hdots,V_{k-1}\oplus V_k;x_1,\hdots,x_{k-1})$ if~$x_{k-1}=x_k$.
\end{itemize}
In words, the first relation says that the labeled configurations are unordered, the second that labels on the basepoint vanish, the third that a point labeled~$0$ can be left out and the last that if two points collide, their labels are added up.

Now we consider the case where~$A=(S^{\sigma})^{\times n}$ with~$S^{\sigma}$ the unit sphere in~$\C$, based at~$1$. So each point~$x_i$ in a labeled configuration~$x=(V_1,\hdots,V_k;x_1,\hdots,x_k)$ has components~$x_i^{(1)},x_i^{(2)},\hdots,x_i^{(n)} \in S^{\sigma}$. For~$j=1,\hdots,n$ we write~$A_j(x)\in U$ for the unique unitary matrix which acts by multiplication with~$x_i^{(j)}$ on each~$V_i$ and equals the identity on the orthogonal complement of the sum of the~$V_i$. Note that for any~$j,j'$ the matrices~$A_j(x)$ and~$A_{j'}(x)$ commute, since they are simultaneously diagonalisable by definition. Furthermore, the~$A_j(x)$ do not depend on the chosen representative for the equivalence class~$[x]$, because one easily checks that~$A_j(-)$ respects the four types of relations above.  Thus we obtain a map
\begin{equation} \label{eq:varphi} \varphi_n\co k\R((S^{\sigma})^{\times n}) \to C_n^{\R}\end{equation}
sending~$[x]=[(V_1,\hdots,V_k;x_1,\hdots,x_k)]$ to the tuple~$(A_1(x),\hdots,A_n(x))\in U^{\times n}$.

\begin{Prop} The map~$\varphi_n$ is a~$C_2$--equivariant homeomorphism.
\end{Prop}
\begin{proof} Equivariance follows from the fact that if a unitary matrix~$A$ acts by multiplication with the scalar~$\lambda$ on~$V$, then its conjugate~$\overline{A}$ acts by multiplication with~$\overline{\lambda}$ on~$\overline{V}$.

To see that~$\varphi_n$ is bijective, we construct an inverse: Let~$A=(A_1,\hdots,A_n)$ be an~$n$--tuple of commuting unitary matrices. Since the~$A_i$ are diagonalisable and commute pairwise, there exist pairwise orthogonal finite dimensional subspaces~$V_1(A),\hdots,V_k(A)$ of~$\C^\infty$ such that each~$A_j$ acts by multiplication with a scalar~$\lambda_i^{(j)}$ on~$V_i(A)$, and all the~$A_j$ act trivially on the orthogonal complement of the direct sum of the~$V_i(A)$. Furthermore, there is a coarsest such decomposition, unique up to permutation, which is characterized by the property that for all~$i\neq i'$ there exists a~$j$ such that~$\lambda_i^{(j)}\neq \lambda_{i'}^{(j)}$. Let~$\psi(A)$ be defined as the class of the tuple~$(V_1(A),\hdots,V_k(A);\lambda_1,\hdots,\lambda_k)$ for this minimal decomposition. By construction, we have~$\varphi_n(\psi(A))=A$. To check the other direction, let~$[x]=[(V_1,\hdots,V_k;x_1,\hdots,x_k)]$ be an element of~$k\R((S^{\sigma})^{\times n})$. Without loss of generality, we can assume that the~$x_i$ are pairwise different non-basepoint elements of~$(S^{\sigma})^{\times n}$ and that all~$V_i$ are non-trivial. Then, by definition, each~$A_j(x)$ acts by scalar multiplication with~$x_i^{(j)}$ on~$V_i$ and for each~$i\neq i'$ there exists a~$j$ such that~$x_i^{(j)}\neq x_{i'}^{(j)}$. In other words, the~$V_i$ satisfy the defining property of the decomposition~$V_i(\varphi_n([x]))$. It follows that~$\psi(\varphi_n([x]))=[x]$ and so~$\varphi_n$ is bijective.

We now show that~$\varphi_n$ is continuous. It is enough to check that each component function~$A_i(-)$ is continuous. The topology on~$k\R((S^{\sigma})^{\times n})$ is that of a quotient of the disjoint union of the spaces~$L(\C^{n_1}\oplus\cdots\oplus \C^{n_k},\C^{m})\times (S^{\sigma})^{\times k}$, which in turn carry the quotient topology via the restriction of~$L(\C^{n_1}\oplus\cdots\oplus \C^{n_k}\oplus \C^{m-\sum n_i},\C^{m})\times (S^\sigma)^{\times k}$ to the first~$k$ summands. Unraveling the definitions, it is hence sufficient to check continuity of the map
\[ L(\C^{n_1}\oplus\cdots\oplus \C^{n_k}\oplus \C^{m-\sum n_i},\C^{m})\times (S^{\sigma})^k\to U(m) \]
sending~$(f,\lambda_1,\hdots,\lambda_k)$ to~$f\circ D^{n_1,\hdots,n_k}(\lambda_1,\hdots,\lambda_k)\circ f^{-1}$. Here,~$D^{n_1,\hdots,n_k}(\lambda_1,\hdots,\lambda_k)$ is the diagonal matrix with each~$\lambda_i$ repeated~$n_i$ times, filled up with~$1$'s at the end. Since~$D^{n_1,\hdots,n_k}(-)$ is a continuous function in the~$\lambda_i$, the claim hence follows from continuity of composition and taking inverses in~$U(m)$.

Finally, to see that~$\varphi_n$ is a homeomorphism, let~$k\R((S^{\sigma})^{\times n})_{\leq m}$ be the subspace of~$k\R((S^{\sigma})^{\times n})$ consisting of all elements which can be represented by tuples~$(V_1,\hdots,V_k;x_1,\hdots,x_k)$ for which each~$V_i$ is a subspace of~$\C^m\subset \C^{\infty}$. Then~$\varphi_n$ restricts to a continuous bijective map
\[ (\varphi_n)_{\leq m}\co k\R((S^{\sigma})^{\times n})_{\leq m}\to C_n(U(m)). \]
The space~$k\R((S^{\sigma})^{\times n})_{\leq m}$ is the continuous image of the compact space 
\[ \bigsqcup_{n_1+\hdots + n_k\leq m;\,n_i>0}L(\C^{n_1}\oplus\cdots\oplus \C^{n_k},\C^m)\times (S^{\sigma})^{\times n}, \]
hence compact. Since~$C_n(U(m))$ is Hausdorff, it follows that the restricted map~$(\varphi_n)_{\leq m}$ is a homeomorphism. Both sides carry the colimit topology, so~$\varphi_n$ is a homeomorphism as well.
\end{proof}

\begin{Remark} The special case~$n=1$ gives a~$C_2$--homeomorphism~$k\R(S^{\sigma})\cong U$, leading to a short proof of Real Bott periodicity. Non-equivariantly, this proof is due to Harris \cite{Har80}. The~$C_2$--equivariant version also appeared in \cite{Sus88}, where it is called a `trivial' proof of Bott periodicity.
\end{Remark}

In order to prove the first part of Theorem \ref{thm:main} using Proposition \ref{prop:veryspecial}, we are hence left to show that the component sets of~$C_n(U)$ and~$C_n(O)$ are abelian groups. The underlying space~$C_n(U)$ is path-connected (as for example follows from the fact that~$(S^1)^{\times n}$ is path-connected and the evaluation of a~$\Gamma$--space on any path-connected space is again path-connected), hence automatically group-complete. Thus it remains to see:
\begin{Lemma} \label{lem:grouplike} The components~$\pi_0(C_n(O))\cong \pi_0(k\R((S^{\sigma})^{\times n})^{C_2})$ form an abelian group.
\end{Lemma}
\begin{proof} The abelian monoid structure is given by the effect of the composite
\[ k\R((S^{\sigma})^{\times n})^{\times 2}\xleftarrow{\simeq} k\R((S^{\sigma})^{\times n}\vee (S^{\sigma})^{\times n}) \xrightarrow{k\R(\text{fold})} k\R((S^{\sigma})^{\times n}) \]
on~$\pi_0^{C_2}$, where the first arrow is a~$C_2$--equivariant homotopy inverse to the map~$P_2$ of Lemma \ref{lem:special}. As explained in the proof of the lemma, such a homotopy inverse can be defined by choosing a~$C_2$--equivariant unitary embedding~$\phi\co \C^2\otimes \C^{\infty}\hookrightarrow \C^{\infty}$ (with~$C_2$ acting only through complex conjugation on~$\C^{\infty}$ and not on~$\C^2$). The resulting multiplication is then given by
\begin{multline*} \mu([(V_1,\hdots,V_k;x_1,\dots,x_n)],[(W_1,\hdots,W_l;y_1,\dots,y_l)])= \\ [(\phi(e_1\otimes V_1),\hdots,\phi(e_1\otimes V_k),\phi(e_2\otimes W_2),\hdots,\phi(e_2\otimes W_l);x_1,\hdots,x_k,y_1,\hdots,y_l)].\end{multline*}
Here, it does not matter which~$\phi$ we choose, the induced map on~$\pi_0^{C_2}$ is always the same. Now let~$[(V_1,\hdots,V_k;x_1,\dots,x_k)]$ be~$C_2$--fixed. Without loss of generality, we can assume that the~$x_i$ are pairwise different and all~$V_i$ non-trivial. Then~$C_2$ acts on the set~$\{V_1,\hdots,V_k\}$, i.e., each~$V_i$ is either fixed under conjugation or its conjugate~$\overline{V_i}$ is also contained in this set. If~$V_i=\overline{V_i}$, then the tuple~$(V_i,x_i)$ also lies in~$k\R((S^{\sigma})^{\times n})^{C_2}$. If~$\overline{V_i}=V_j$ for~$j\neq i$, then the tuple~$(V_i,V_j;x_i,x_j)$ lies in~$k\R((S^{\sigma})^{\times n})^{C_2}$. Moreover, the component of~$(V_1,\hdots,V_k;x_1,\dots,x_k)$ is the sum of the components of those tuples, which one can see by induction on the number of the tuples by choosing the map~$\phi$ above accordingly in the induction step. Hence, we can reduce to two special cases: The tuples~$(V,\overline{V};x,\overline{x})$ where~$V$ is such that it is orthogonal to its conjugate~$\overline{V}$, and the tuples~$(V,x)$ for~$V=\overline{V}$ and~$x\in \{-1,1\}$.

Case 1:~$(V,\overline{V};x,\overline{x})$. For every component~$i\in \{1,\hdots,n\}$ we can choose a path from~$x^{(i)}$ to~$1$ inside~$S^{\sigma}$. By choosing the conjugate path for the component~$\overline{x}^{(i)}$, we obtain a path in~$k\R((S^{\sigma})^{\times n})^{C_2}$ from~$(V,\overline{V};x,\overline{x})$ to~$(V,\overline{V};1,1)$, which is equivalent to the basepoint. Hence,~$(V,\overline{V};x,\overline{x})$ already represents the path component of the basepoint.

Case 2:~$(V,x)$. We can further reduce to~$V$ being~$1$--dimensional, since otherwise we can write the class of~$(V,x)$ as the sum of even smaller tuples by choosing an orthogonal decomposition of~$V$ into~$1$--dimensional subspaces. Let~$W$ be any~$1$--dimensional subspace of~$\C^{\infty}$ which is orthogonal to~$V$ and satisfies~$\overline{W}=W$. We claim that~$(W,x)$ is an inverse for~$(V,x)$, up to homotopy. By making a suitable choice for the map~$\phi$ above, one sees that their sum is represented by the tuple~$(V,W;x,x)$, which is equivalent to~$(V\oplus W; x)$. This tuple corresponds to an element of~$SO(V\oplus W)^{\times n}\subset C_n(O)$. Since~$SO(V\oplus W)^{\times n}\cong (S^1)^{\times n}$ is path-connected, we are done.
\end{proof}
This finishes the proof that there is a~$C_2$--equivalence
\begin{equation} \label{eq:hom} C_n^{\R}\cong k\R((S^{\sigma})^{\times n})\simeq \Omega^{\infty} (k\R\wedge (S^{\sigma})^{\times n}). \end{equation}
Next we consider the quotient~$\Rep_n^{\R}=C_n^{\R}/U$. We first justify that the~$C_2$--fixed points of~$\Rep_n^{\R}$ indeed agree with~$\Rep_n(O)$, which is not completely formal.

\begin{Lemma}
The inclusion~$C_n(O)\to C_n^{\R}$ induces a homeomorphism~$\Rep_n(O)\cong (\Rep_n^{\R})^{C_2}$.
\end{Lemma}
\begin{proof}
Fixed points commute with sequential colimits along closed inclusions, so it suffices to show that for every~$k\geq 1$ the map~$c_k\colon \thinspace \Rep_n(O(k))\to (\Rep_n(U(k)))^{C_2}$ is a homeomorphism. In fact, it suffices to show that~$c_k$ is bijective, since~$\Rep_n(O(k))$ is compact and~$(\Rep_n(U(k)))^{C_2}$ is Hausdorff.

Injectivity of~$c_k$ follows from the general facts that if~$V$ and~$W$ are real representations of an arbitrary group~$G$ and~$\C\otimes_{\R}V\cong \C\otimes_{\R} W$ as complex~$G$--representations, then~$V\cong W$ as real~$G$--representations (\cite[Corollary 3.28]{Ad69}), and two orthogonal~$G$--representations which are isomorphic as real~$G$--representations are already isomorphic as orthogonal~$G$--representations (see e.g.  \cite[Lemma 1.9]{Go12}).

To prove that~$c_k$ is surjective, suppose that~$V\in \Rep_n(U(k))$ is fixed under the~$C_2$--action. This means that there is an isomorphism~$V\cong V^c$, where~$V^c$ is the conjugate representation. We must show that~$V$ is the complexification of an orthogonal representation. Since~$\Z^n$ is abelian, we may assume without loss of generality that~$V$ is a direct sum~$V_1\oplus \dots \oplus V_k$ of~$1$--dimensional unitary representations. Then for each~$i$ we have that either~$V_i\cong V^c_i$ or~$V_i\cong V^c_j$ for some~$j\neq i$. In the first case, an isomorphism~$V_i\cong V^c_i$ may be viewed equivalently as a conjugate linear isomorphism~$J\colon \thinspace V_i\to V_i$. Now~$J$ being unitary and conjugate linear implies that~$J^2=\id$, i.e.,~$J$ is a~$\Z^n$--equivariant Real structure on~$V_i$. It follows that~$V_i$ is isomorphic to the complexification of the~$(+1)$--eigenspace of~$J$, which is an orthogonal~$1$--dimensional representation. Now if~$V_i\cong V^c_j$ for~$j\neq i$, then~$V_i\oplus V_j\cong V_i\oplus V^c_i$. In this case, fix a non-zero~$e\in V_i$ and let~$\overline{e}\in V^c_i$ be the same vector viewed as an element in the conjugate representation. Define~$J\colon \thinspace V_i\oplus V^c_i\to V_i\oplus V^c_i$ by~$J(e)=\overline{e}$,~$J(\overline{e})=e$ and conjugate linearity. Then~$J$ is a~$\mathbb{Z}^n$--equivariant Real structure on~$V_i\oplus V^c_i$ and~$V_i\oplus V^c_i$ is isomorphic to the complexification of the~$(+1)$--eigenspace of~$J$. This is an orthogonal~$2$--dimensional representation. Altogether, we see that~$c_k$ is surjective.\end{proof}

Now we finish the proof of Theorem \ref{thm:main}. By what we saw above,~$C_n^{\R}$ can be described as configurations on~$(S^{\sigma})^{\times n}$ with labels in pairwise orthogonal subspaces of~$\C^{\infty}$. In this picture, the unitary group~$U$ acts through the labels and fixes the configurations. Now, given any two tuples of pairwise orthogonal subspaces~$(V_1,\hdots,V_k)$ and~$(V'_1,\hdots,V'_k)$, there exists an element~$A\in U$ such that~$A(V_i)=V'_i$ for all $i=1,\hdots,k$ if and only if~$\dim(V_i)=\dim(V'_i)$. This means that the quotient space~$\Rep_n^{\R}$ only remembers the dimension of the labels. In other words, it can be described as configurations in~$(S^{\sigma})^{\times n}$ with labels in the natural numbers, i.e., it is homeomorphic to the infinite symmetric product~$Sp^{\infty}((S^{\sigma})^{\times n})$. Under this homeomorphism, the filtration by the~$\Rep_n(U(k))$ corresponds to the filtration of~$Sp^{\infty}((S^{\sigma})^{\times n})$ by the finite symmetric products~$Sp^k((S^{\sigma})^{\times n})$. It follows from a theorem of dos Santos \cite{dS03} that the infinite symmetric product~$Sp^{\infty}(A)$ of a based~$C_2$--CW complex~$A$ agrees with~$\Omega^{\infty}(H\underline{\Z}\wedge A)$ if the component sets~$\pi_0(Sp^{\infty}(A)^{C_2})$ and~$\pi_0(Sp^{\infty}(A))$ are groups. Using~$C_2-\Gamma$--spaces, this statement follows from Propositions \ref{prop:special} and \ref{prop:veryspecial} by noting that~$Sp^{\infty}(A)$ is the evaluation of the~$C_2$--cofibrant special~$C_2-\Gamma$--space
\[ k_+\mapsto Sp^{\infty}(k_+). \]
The infinite loop space of the spectrum realization of this~$C_2-\Gamma$--space is given by
\[ \Omega^{\infty}(Sp^{\infty}(\mathbb{S}))\simeq \Omega Sp^{\infty}(S^1)\simeq \Omega S^1 \simeq \Z \]
with trivial~$C_2$--action, using the fact that the infinite symmetric product of~$S^1$ is again equivalent to~$S^1$ by the Dold-Thom theorem. Hence,~$Sp^{\infty}(\mathbb{S})$ is a model for~$H\underline{\Z}$. Note that the component sets~$\pi_0(\Rep_n(U))$ and~$\pi_0(\Rep_n(O))$ are groups, since this is the case for the spaces of commuting elements and the map~$C_n^{\R}\to \Rep_n^{\R}$ induces a surjective monoid homomorphism on components and on components of fixed points. Thus we obtain an equivalence
\begin{equation} \label{eq:rep} \Rep_n^{\R} \cong Sp^{\infty}((S^{\sigma})^{\times n})\simeq \Omega^{\infty}(H\underline{\Z}\wedge (S^{\sigma})^{\times n}), \end{equation}
as desired. Moreover, under the two equivalences \eqref{eq:hom} and \eqref{eq:rep}, the map~$C_n^{\R}\to \Rep_n^\R$ is induced by the map of~$C_2-\Gamma$--spaces~$k\R\to Sp^{\infty}(-)$ sending~$(V_1,\hdots,V_k)$ to~$(\dim V_1,\hdots,\dim V_k)$, which models the~$0$--th Postnikov section after spectrum realization.
This finishes the proof of Theorem~\ref{thm:main}.\bigskip

Finally, we come to the proof of Theorem \ref{thm:bcom}, i.e., the identification of the homotopy type of~$B_{\textnormal{com}}^{\R}$. Recall that restricting the structure maps of the bar construction turns the collection of~$C_2$--spaces~$C_\bullet^{\R}$ into a simplicial~$C_2$--space. On the other hand,~$S^{\sigma}$ is a~$C_2$--equivariant topological group via multiplication in~$\C$, and so the spaces~$(S^{\sigma})^{\times n}$ also form a simplicial~$C_2$--space via the bar construction. Applying~$k\R(-)$ we obtain a simplicial~$C_2$--space~$k\R((S^{\sigma})^{\times \bullet})$.
\begin{Lemma} \label{lem:simplicial} The homeomorphisms~$\varphi_n\co k\R((S^{\sigma})^{\times n})\xrightarrow{\cong} C_n^{\R}$ assemble into an isomorphism of simplicial~$C_2$--spaces
\[ \varphi:k\R((S^{\sigma})^{\times \bullet})\xrightarrow{\cong} C_{\bullet}^{\R}. \]
\end{Lemma}
\begin{proof} This is a direct check which comes down to the fact that if~$A$ and~$A'$ act by multiplication with~$\lambda$ respectively~$\lambda'$ on a subspace~$V$ of~$\C^\infty$, then~$A\cdot A'$ acts by multiplication with~$\lambda\cdot \lambda'$ on~$V$. 
\end{proof}

It follows that~$\varphi$ induces a~$C_2$--equivariant homeomorphism 
\[ |k\R((S^{\sigma})^{\times \bullet})|\xrightarrow{\cong} |C_{\bullet}^{\R}|=B_{\textnormal{com}}^{\R}. \]
Furthermore, the evaluation functor~$k\R(-)$ commutes with geometric realization (\cite[Proposition B.29]{Sch18}), hence we obtain a~$C_2$--equivariant homeomorphism
\[ |k\R((S^{\sigma})^{\times \bullet})|\cong k\R(|(S^{\sigma})^{\times \bullet}|). \]
The~$C_2$--space~$BS^{\sigma}=|(S^{\sigma})^{\times \bullet}|$ is~$C_2$--equivariantly homotopy equivalent to~$\CP^{\infty}$ with action by complex conjugation, because both are classifying spaces for the~$C_2$--group~$S^{\sigma}$. Since~$k\R(-)$ preserves~$C_2$--homotopy equivalences, we have proved that~$B_{\textnormal{com}}^{\R}$ is~$C_2$--equivalent to~$k\R(\CP^{\infty})$. Now both the underlying space and the~$C_2$--fixed points of~$\CP^{\infty}$ are path-connected, hence the same is true for the evaluation~$k\R(\CP^{\infty})$. In particular, Proposition \ref{prop:veryspecial} applies, which finishes the proof of Theorem \ref{thm:bcom}.

\begin{Remark} The fixed point space~$B_{\textnormal{com}}O\simeq \Omega^{\infty}(k\R\wedge \CP^{\infty})^{C_2}$ receives a map
\[ \Omega^{\infty}(ko\wedge \RP^{\infty})\simeq \Omega^{\infty}(k\R\wedge \RP^{\infty})^{C_2} \to \Omega^{\infty}(k\R\wedge \CP^{\infty})^{C_2}, \]
induced by the inclusion~$\RP^{\infty}\to \CP^{\infty}$, viewed as a~$C_2$--equivariant map by giving~$\RP^{\infty}$ the trivial action. Using the methods from this section, one can show that~$\Omega^{\infty}(ko\wedge \RP^{\infty})$ is in fact equivalent to the subspace of~$B_{\textnormal{com}}O$ given by the geometric realization of the simplicial space of commuting tuples of symmetric orthogonal matrices
\[ n\mapsto \{ (A_1,\hdots,A_n)\in C_n(O)\ |\ A_i=A_i^t \text{ for all }i\}, \]
and the map above corresponds to the inclusion into~$B_{\textnormal{com}}O$.
\end{Remark}

\begin{Remark} There is a different way to package the various $C_n^{\R}$ into a single $C_2$--space, denoted $\Comm^{\R}$, that has been considered non-equivariantly in \cite{CS16b}. It is defined via
\[ \Comm^{\R} = (\bigsqcup_{n\in \mathbb{N}} C_n^{\R})/\sim \]
with the equivalence relation generated by $(A_1,\hdots,A_n)\sim (A_1,\hdots,A_{i-1},A_{i+1},\hdots,A_n)$ if $A_i$ is the identity matrix. In other words, $\Comm^{\R}$ is the colimit along the degeneracies $\Delta_{surj}^{op}$ in the simplicial $C_2$--space $C_{\bullet}^{\R}$ above. The $C_2$--space $\Comm^{\R}$ is a subspace of the reduced James construction $J(U)$ which is obtained via the same formulas, replacing $C_n^{\R}$ by the full cartesian product $U^{\times n}$. Up to weak equivalence, $\Comm^{\R}$ can also be defined as the homotopy colimit of the functor $\Delta_{surj}^{op}\to C_2-spaces$ (this can be derived from \cite[Proposition A.1]{AG15}, which says that in every degree the inclusion of the degenerate simplices is a Hurewicz cofibration both on underlying spaces and on $C_2$--fixed points). Hence, Theorem \ref{thm:main} and Lemma \ref{lem:simplicial} provide a $C_2$--equivalence
\[ \Comm^{\R} \simeq \hocolim_{[n]\in \Delta_{surj}^{op}} \Omega^{\infty}(k\R \wedge (S^{\sigma})^{\times n}),\]
where the structure maps on the right hand side are induced by the degeneracies in the bar construction for $S^{\sigma}$. In particular, these structure maps are infinite loop maps, but the homotopy colimit is taken inside $C_2$--spaces and not infinite loop $C_2$--spaces. Unlike geometric realization, homotopy colimits indexed on $\Delta_{surj}^{op}$ do not commute with the forgetful functor from infinite loop $C_2$--spaces to $C_2$--spaces. As a consequence, the infinite loop space structures on the $C_n^{\R}$ do not induce an infinite loop space structure on $\Comm^{\R}$. If one forms the homotopy colimit in infinite loop $C_2$--spaces instead, one obtains
\[\Omega^{\infty} (\hocolim_{[n]\in \Delta_{surj}^{op}}k\R\wedge (S^{\sigma})^{\times n})\simeq \Omega^{\infty}(k\R\wedge \hocolim_{[n]\in \Delta_{surj}^{op}}(S^{\sigma})^{\times n})\simeq \Omega^{\infty}(k\R\wedge J(S^{\sigma})) ,\] where $J(S^{\sigma})$ is the reduced James construction on $S^{\sigma}$.
\end{Remark}

\section{Homotopy groups} \label{sec:homgroups}
In this section we compute the homotopy groups of the spaces of commuting elements~$C_n(U)$ and~$C_n(O)$, the spaces of representations~$\Rep_n(U)$ and~$\Rep_n(O)$, and the classifying spaces~$B_{\textnormal{com}}U$ and~$B_{\textnormal{com}}O$ using Theorems \ref{thm:main} and \ref{thm:bcom} and well-known descriptions of the homotopy groups of~$k\R$ and~$H\underline{\Z}$. Our main focus is on the more subtle orthogonal case, but we go through the unitary case as a warm-up. The presentation of the homotopy groups is most compact when stated as a functor in~$n$, which is described in Section \ref{sec:fi}.

\subsection{\texorpdfstring{$k\R$}{kR}--module structure} \label{sec:krmodule} We will use the fact that $k\R$ and $H\underline{\mathbb{Z}}$ are homotopy-commutative~$C_2$--ring spectra, which is well-known. In fact they allow strictly commutative multiplications, as can also be seen with the models that we use: The infinite symmetric product $Sp^{\infty}$ carries a commutative multiplication via 
\begin{alignat*}{2} && Sp^{\infty}(S^V)\wedge Sp^{\infty}(S^W)& \to Sp^{\infty}(S^{V\oplus W}) \\
&& (x_i;a_i)_{i=1,\hdots,n}\wedge (y_j;b_j)_{j=1,\hdots,m} & \mapsto (x_i\wedge y_j;a_ib_j)_{i=1,\hdots,n, j=1,\hdots,m}\end{alignat*}
with unit $(*,1)\in Sp^{\infty}(S^0)$, where $*\in S^0$ denotes the non-basepoint. The formula for $k\R$ is similar by taking the tensor product of the labeling complex subspaces. To make this work, one needs to modify the model for $k\R$ slightly, since the tensor product of two subspaces of $\C^{\infty}$ is not canonically again a subspace of $\C^{\infty}$. One way is to replace the ambient $\C^{\infty}$ by~$\Sym(V_{\C})$, the free commutative algebra on the complexification of $V$, varying functorially in the levels $V$ of the orthogonal $C_2$--spectrum. This works because the tensor product of a subspace of~$\Sym(V_{\C})$ with a subspace of~$\Sym(W_{\C})$ is a subspace of~$\Sym(V_{\C})\otimes \Sym(W_{\C})$, which is naturally isomorphic to~$\Sym((V\oplus W)_{\C})$. Hence, the assignment $(x_i;V_i)_i\wedge(y_j;W_j)_j\mapsto (x_i\wedge y_j;V_i\otimes W_j)_{i,j}$ gives a well-defined commutative multiplication with unit $(*,\C)$.
See \cite[Section 6.3]{Sch18} for the details on this construction in the case of equivariant K-theory.

Because of the simple structure of~$k\R\wedge (S^{\sigma})^{\times n}$ and~$k\R\wedge \CP^\infty$ as~$k\R$--module spectra, deriving the homotopy groups from our theorems is straightforward. We first recall that for any two based~$C_2$--CW complexes $A$ and $B$, the cofiber sequence
\[ A\vee B\to A\times B \to A\wedge B \]
has a natural stable splitting $A\times B\to A\vee B$, obtained by taking the sum over the two projections. In particular, there is a natural $C_2$--stable equivalence
\[ \Sigma^{\infty}(A\times B)\simeq \Sigma^{\infty}(A\vee B \vee (A\wedge B))\,. \]
Setting $A=B=S^{\sigma}$ and iterating this decomposition, we obtain a $C_2$--stable equivalence
\[ \Sigma^{\infty}\left((S^{\sigma})^{\times n}\right)\simeq \Sigma^{\infty}\left(\bigvee_{k=1}^n\bigvee_{n \choose k} S^{k\cdot \sigma}\right)\,. \]
As a consequence, there is a splitting of~$k\R$--modules
\[
k\R\wedge (S^{\sigma})^{\times n}\simeq \bigvee_{k=1}^n\bigvee_{n \choose k} k\R\wedge S^{k\cdot \sigma} \,.
\]
Hence, by Theorem \ref{thm:main}, we obtain a~$C_2$--equivalence
\begin{equation} \label{eq:splitcnr}
C_n^{\R}\simeq \prod_{k=1}^n\prod_{n\choose k} \Omega^{\infty}\left(k\R\wedge S^{k\cdot \sigma}\right)\, .
\end{equation}
Similarly, there is a~$C_2$--equivalence
\begin{equation} \label{eq:splitrnr}
\Rep_n^{\R}\simeq \prod_{k=1}^n\prod_{n\choose k} \Omega^{\infty}\left(H\underline{\Z}\wedge S^{k\cdot \sigma}\right)\,.
\end{equation}

The~$C_2$--space~$\CP^{\infty}$ does not decompose stably into its cells. However, it does split after smashing with~$k\R$, because~$k\R$ is Real orientable. A Real orientation is given by the canonical~$C_2$--equivariant map $\CP^{\infty}\simeq BS^{\sigma}\to BU\simeq \Omega^{\infty}(k\R\wedge S^{1+\sigma})$, which restricts to the unit on~$\CP^1\cong S^{1+\sigma}$ (cf. \cite[Theorem 2.8]{HK01}). By \cite[Theorem 2.10]{HK01}, this determines an equivalence of~$k\R$--modules
\[ k\R\wedge \CP^{\infty}\simeq \bigvee_{k>0}k\R\wedge S^{k\cdot (1+\sigma)}. \]
It follows from Theorem \ref{thm:bcom} that there is a~$C_2$--equivalence
\begin{equation} \label{eq:splitbcom}
B_{\textnormal{com}}{\R}\simeq {\prod_{k>0}}'\Omega^{\infty} (k\R\wedge S^{k\cdot (1+\sigma)} )\,,
\end{equation}
where the decoration indicates that the weak infinite product is used.

\subsection{Unitary case}

\subsubsection{The homotopy groups of $C_n(U)$} Non-equivariantly, the equivalence (\ref{eq:splitcnr}) shows that~$C_n(U)$ factors as a product of infinite loop spaces of the form~$\Omega^{\infty}(ku\wedge S^k)$ for certain~$k$. By Bott periodicity, this infinite loop space is~$U\langle k\rangle$ when~$k$ is odd and~$BU\langle k\rangle$ when~$k$ is even, where~$X\langle k\rangle$ denotes the~$(k-1)$--connected cover of a space~$X$. The graded group~$\pi_*(\Omega^{\infty}(ku\wedge S^k))$ is thus free as a module over~$\pi_\ast(ku)$ on one generator in degree~$k$. Let us denote this module by~$\pi_\ast(ku)[k]$. Then we have an isomorphism
\[ \pi_*(C_n(U))\cong \bigoplus_{k=1}^n \bigoplus_{n\choose k} \pi_\ast(ku)[k] \]
of~$\pi_\ast(ku)$--modules. 
Counting ranks we obtain
\[ \pi_{2i}(C_n(U))\cong \Z^{\sum_{1\leq j\leq i} {n\choose 2j}}\,, \hspace{0.8cm} \pi_{2i+1}(C_n(U))\cong \Z^{\sum_{0\leq j\leq i} {n\choose 2j+1}}\]
for all~$i\geq 0$.

\subsubsection{The homotopy groups of $\Rep_n(U)$} Similarly, on underlying spaces, the equivalence (\ref{eq:splitrnr}) describes a decomposition of~$\Rep_n(U)$ into a product of Eilenberg-MacLane spaces. The homotopy type of~$\Rep_n(U)$ was previously described by Ramras \cite[Theorem 6.6]{Ra11} and Adem-Cohen-G\'omez \cite[Theorem 6.8]{ACG10}. There are isomorphisms
\[ \pi_i(\Rep_n(U))\cong \Z^{n\choose i}\]
for all~$i \geq 1$. Moreover, since~$C_n(U)\to \Rep_n(U)$ is induced by the Postnikov section~$ku \to H\Z$, the induced map of graded groups
\[ \pi_*(C_n(U))\cong \bigoplus_{k=1}^n \bigoplus_{n\choose k} \pi_\ast(ku)[k]\to \pi_*(\Rep_n(U))\cong \bigoplus_{k=1}^n \bigoplus_{n\choose k} \Z[k] \]
is given by collapsing~$\pi_\ast(ku)$ to the copy of~$\Z$ in degree~$0$.

\subsubsection{The homotopy ring of $B_{\textnormal{com}}U$} \label{sec:bcomu} The classifying space for commutative complex K--theory $B_{\textnormal{com}}U$ was the main subject of \cite{Gri18}. Using (\ref{eq:splitbcom}) we see that~$B_{\textnormal{com}}U$ splits as a product of all the higher connected covers~$BU\langle 2k\rangle$ for~$k>0$. On the level of spectra this is a splitting of~$ku$--modules, so there is an isomorphism
\[
\pi_\ast(B_{\textnormal{com}}U)\cong \bigoplus_{k>0} \pi_\ast(ku)[2k]
\]
of~$\pi_\ast(ku)$--modules. In particular, the homotopy groups of~$B_{\textnormal{com}}U$ are
\[
\pi_{2i}(B_{\textnormal{com}}U)=\Z^i\,, \hspace{0.8cm} \pi_{2i+1}(B_{\textnormal{com}}U)=0
\]
for all~$i\geq 0$.

The spectrum $ku\wedge \CP^\infty_+$, the $ku$--group ring of $\CP^{\infty}$, is an augmented $ku$-algebra. Consequently,~$\pi_*(ku\wedge \CP^\infty_+)$ is an augmented~$\pi_*(ku)$--algebra. The~$\pi_\ast(ku)$--module~$\pi_\ast(B_{\textnormal{com}}U)\cong \pi_\ast(ku\wedge \CP^\infty)$ is canonically isomorphic to the augmentation ideal of this augmented~$\pi_*(ku)$-algebra and hence carries the structure of a non-unital~$\pi_\ast(ku)$--algebra. It seems worth spelling out this multiplicative structure, since on the level of~$B_{\textnormal{com}}U$ it admits a geometric description in terms of tensor products of vector bundles. This also serves as a preparation for our discussion of the ring structure of~$\pi_\ast(B_{\textnormal{com}}O)$ in Section \ref{sec:bcomo}.

The non-reduced algebra $\pi_*(ku\wedge \CP^\infty_+)$ is best described using the work of Ravenel and Wilson \cite{RW77}. The canonical map~$\CP^\infty\to BU$ into the second space of the spectrum~$ku$ represents a complex orientation class~$x\in ku^2(\CP^\infty)$. Let~$y_i\in ku_{2i}(\CP^\infty)$ denote the dual~$ku$--homology class of~$x^{i}\in ku^{2i}(\CP^\infty)$. Then the Atiyah-Hirzebruch spectral sequence shows that~$\pi_\ast(ku\wedge\CP^\infty)$ is a free~$\pi_\ast(ku)$--module on the generators~$y_i$ for~$i\geq 0$. In the algebra~$\pi_\ast(ku\wedge\CP^\infty_+)$ the module generators~$y_i$ obey certain relations, which are described by \cite[Theorem 3.4]{RW77}. Consider the formal power series
\[
y(t)=\sum_{i\geq 0} y_i t^i \in \pi_\ast(ku\wedge \CP^\infty_+)[[t]]\, .
\]
Let~$v\in \pi_2(ku)$ be the Bott class. Then the identity
\begin{equation} \label{eq:yy}
y(s)y(t)=y(s+t+vst)
\end{equation}
holds in~$\pi_\ast(ku\wedge \CP^\infty_+)[[s,t]]$. Let~$R$ denote the set of relations amongst the~$y_i$ following from this identity.
Then the canonical map~$\pi_\ast(ku)[\{y_i\}]\to \pi_\ast(ku\wedge \CP^\infty_+)$ induces an isomorphism
\[
\pi_\ast(ku)[\{y_i\}]/R\xr{\cong} \pi_\ast(ku\wedge \CP^\infty_+)
\]
of~$\pi_\ast(ku)$--algebras. In this picture, the ring~$\pi_\ast(B_{\textnormal{com}}U)$ appears as the ideal in~$\pi_\ast(ku\wedge \CP^\infty_+)$ generated by all the~$y_i$ for~$i\geq 1$. It is not difficult to extract from (\ref{eq:yy}) an explicit expression for the product of two classes~$y_k$ and~$y_l$.
One finds that
\begin{equation} \label{eq:ykyl}
y_k y_l=\sum_{i=\max\{k,l\}}^{k+l} \frac{i!}{(i-k)! (i-l)!(k+l-i)!} v^{k+l-i}y_i
\end{equation}
for all~$k,l\geq 0$.

\subsection{Orthogonal case} \label{sec:orthogonal}
Before we describe the homotopy groups of~$C_n(O)$, ~$\Rep_n(O)$ and $B_{\textnormal{com}}O$, we recall some properties of~$RO(C_2)$--graded homotopy groups. For a~$C_2$--spectrum~$X$ and integers~$k,l$ one defines
\[ \pi^{C_2}_{k+l\cdot \sigma}(X)=[S^{k+l\cdot \sigma},X]^{C_2}, \]
where for negative~$k$ and~$l$ one uses the fact that the spheres~$S^1$ and~$S^{\sigma}$ are invertible objects in the genuine~$C_2$--equivariant stable homotopy category. The collection of the groups~$\pi^{C_2}_{k+l\cdot \sigma}(X)$ is often denoted~$\pi_{\Star}^{C_2}(X)$. If~$X$ is a~$C_2$--ring spectrum up to homotopy with multiplication~$\mu$, then~$\pi_{\Star}^{C_2}(X)$ becomes a graded ring, where the product of two classes~$\alpha\in \pi_{k_1+l_1\cdot \sigma}^{C_2}(X)$ and~$\beta\in \pi_{k_2+l_2\cdot \sigma}^{C_2}(X)$ is defined as the composite
\[ S^{k_1+k_2+l_1\cdot \sigma +l_2\cdot \sigma}\cong S^{k_1+l_1\cdot \sigma +k_2+l_2\cdot \sigma}\cong S^{k_1+l_1\cdot \sigma}\wedge S^{k_2 +l_2\cdot \sigma}\xrightarrow{\alpha\wedge \beta} X\wedge X\xrightarrow{\mu} X\, . \]
Forgetting $C_2$--equivariance gives restriction maps $\pi_{k+l\cdot \sigma}^{C_2}(X)\to \pi_{k+l}(X)$ for all $k,l$, which are compatible with the ring structures.

If the multiplication on~$X$ is homotopy-commutative, the ring~$\pi_{\Star}^{C_2}(X)$ is graded commutative in the equivariant sense, which is a little subtle in general as it involves involutions in the Burnside ring (see \cite[Lemma 2.12]{HK01}). However, for~$k\R$ and~$H\underline{\Z}$ the~$RO(C_2)$--graded homotopy ring turns out to be strictly commutative. We give some details in Section \ref{sec:bcomo} below, where we argue that the non-unital homotopy ring~$\pi_{\Star}^{C_2}(k\R\wedge \CP^{\infty})$ is also strictly commutative.

\subsubsection{The homotopy groups of $C_n(O)$} \label{sec:cno} The~$RO(C_2)$--graded ring~$\pi_{\Star}^{C_2}(k\R)$ is well-understood, due to work of Dugger \cite{Du05}, Bruner-Greenlees \cite{BG10} and others. We recommend the survey paper \cite{Gre18}, which discusses different approaches to this calculation and the one of~$\pi_{\Star}^{C_2}(H\underline{\Z})$. Figures \ref{figure:kr} and \ref{figure:hz} are taken from there. For us it is enough to consider the part of~$\pi_{k+l\cdot \sigma}^{C_2}(k\R)$ where~$k\geq l$. This subring is generated by four elements:
\begin{itemize}
	\item The Bott class~$\overline{v}$ in degree~$(1+\sigma)$.
	\item The Euler class~$a$ in degree~$-\sigma$ given by the composite
	\[ S^0\to k\R\wedge S^0 \to k\R\wedge S^{\sigma}, \]
	where the first map is the unit of~$k\R$ and the second map is induced by the inclusion of fixed points~$S^0\hookrightarrow S^{\sigma}$.
	\item The classes~$w$ in degree~$(2-2\cdot \sigma)$ and~$U$ in degree~$(4-4\cdot \sigma)$.
\end{itemize}
Here we have used the notation from Greenlees \cite{Gre18}, except for~$w$, which is called~$2u$ there. The relations these four classes satisfy are generated by~$2a=0$,~$aw=0$,~$a^3\overline{v}=0$ and~$w^2=4U$. See Figure \ref{figure:kr} for a drawing of the homotopy groups of~$k\R$. The horizontal axis is a copy of the ring~$\pi_\ast(ko) \cong \Z[\eta,x,y]/(2\eta, \eta x, \eta^3, x^2-4y)$, where~$\eta=a\bar{v}$ generates~$\pi_1(ko)$,~$x=w \bar{v}^2$ generates~$\pi_4(ko)$ and~$y=U\bar{v}^4$ generates~$\pi_8(ko)$.

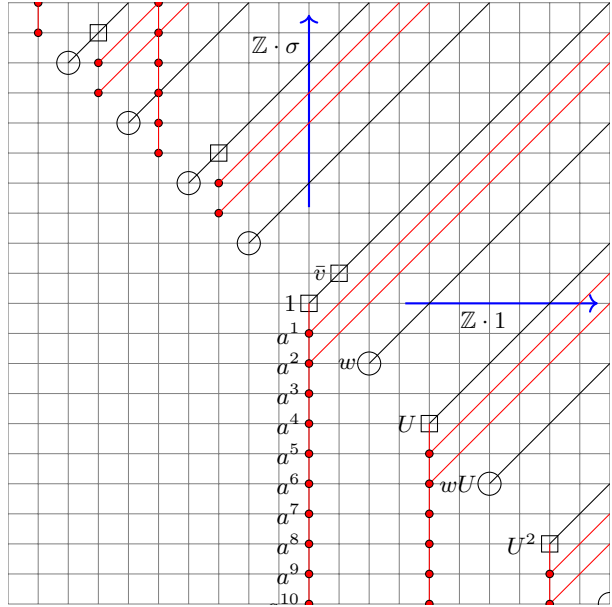
\begin{figure}[h]
\begin{tikzpicture}[scale =0.8]
\small
\draw (1.5,1.5) node[anchor=east]{$\bar{v}\,$}; % Bott-Klasse

\node at (1.5,1.5)  [shape = rectangle, draw]{};
\clip (-4, -4) rectangle (6, 6); %Ausschneiden
\draw[step=0.5, gray, very thin] (-4,-4) grid (6, 6);

\draw[blue, ->, thick](2.6,1)--(5.8, 1); % horizontale Achse / Pfeil
\draw (3.9, 1) node[anchor=north]{$\Z\cdot 1$};

\draw[blue, ->, thick](1,2.6)--(1, 5.8); % vertikale Achse / Pfeil

\draw [-](2,0)-- (6, 4);
\node at (2,0) [shape=circle, draw] {};
\draw (2,0) node[anchor=east]{$w\,$};

\draw [-] (1,1)-- (6, 6);
\node at (1,1)  [shape = rectangle, draw]{};
\draw (1,1) node[anchor=east]{$1\,$};

\draw[red] (1,1) -- (1,-4);

%\draw[red] (0,0)-- (1,-1);

\draw [-] (3,-1)-- (6, 2);
\node at (3,-1)  [shape = rectangle, draw]{};
\draw (3,-1) node[anchor=east]{$U\,$};

\draw [-](4,-2)-- (6, 0);
\node at (4,-2) [shape=circle, draw] {};
\draw (4,-2) node[anchor=east]{$wU\,$};

\draw[red] (1,0.5)-- (6, 5.5);
\draw[red] (1,0)-- (6, 5);
\foreach \x in {1, 2, 3,4,5,6,7,8,9,10}
\draw (1,1-\x/2) node[anchor=east] {$a^{\x}$};
\foreach \x in {1,2,3,4,5,6,7,8,9,10}
\node at (1,1-\x/2) [fill=red, inner sep=1pt, shape=circle, draw] {};

\draw[red] (3,-1) -- (3,-4);

\draw[red] (3,-1.5)-- (6, 1.5);
\draw[red] (3,-2)-- (6, 1);

\foreach \x in {1,2,3,4,5,6}
\node at (3,-1-\x/2) [fill=red, inner sep=1pt, shape=circle, draw] {};

%\draw[red] (2,-2)-- (3,-3);

\draw [-] (5,-3)-- (6, -2);
\node at (5,-3)  [shape = rectangle, draw]{};
\draw (5,-3) node[anchor=east]{$U^2\,$};

\node at (6,-4) [shape= circle, draw] {};

\draw[red] (5,-3) -- (5,-4);
\draw[red] (5,-3.5)-- (6, -2.5);
\draw[red] (5,-4)-- (6, -3);
\foreach \x in {1,2}
\node at (5,-3-\x/2) [fill=red, inner sep=1pt, shape=circle, draw] {};

%\draw[red] (4,-4)-- (5,-5);

%%% The following part is not relevant to us

\draw (-1,3)-- (2,6);
\node at (-1,3) [shape=circle, draw] {};
\node at (-0.5,3.5) [shape=rectangle, draw] {};

\draw (-3,5)-- (-2, 6);
\node at (-3,5) [shape=circle, draw] {};
\node at (-2.5,5.5) [shape=rectangle, draw] {};

\draw (-2,4)-- (0, 6);
\node at (-2,4) [shape=circle, draw] {};

\draw (0,2)-- (4, 6);
\node at (0,2) [shape=circle, draw] {};

\foreach \x in {0,1,2,3,4,5}
\node at (-1.5,3.5+\x/2) [fill=red, inner sep=1pt, shape=circle, draw]
{};
\draw[red] (-1.5, 3.5) --(-1.5, 6);

\node at (-0.5,2.5) [fill=red, inner sep=1pt, shape=circle, draw] {};
\draw[red] (-0.5,2.5)-- (3, 6);
\node at (-0.5,3.0) [fill=red, inner sep=1pt, shape=circle, draw] {};
\draw[red] (-0.5,3.0)-- (2.5, 6);

\foreach \x in {0,1}
\node at (-3.5,5.5+\x/2) [fill=red, inner sep=1pt, shape=circle, draw]
{};
\draw[red] (-3.5, 5.5) --(-3.5, 6);

\node at (-2.5,5) [fill=red, inner sep=1pt, shape=circle, draw] {};
\draw[red] (-2.5,5)-- (-1.5, 6);
\node at (-2.5,4.5) [fill=red, inner sep=1pt, shape=circle, draw] {};
\draw[red] (-2.5,4.5)-- (-1, 6);

\draw (1,5.3) node[anchor=east]{$\Z\cdot \sigma$};

\end{tikzpicture}
\caption{The ring~$\pi_{\Star}^{C_2}(k\R)$. The part which is relevant for our computation lies on the diagonal and below. Black circles, squares and lines represent copies of~$\Z$; red dots and lines represent copies of~$\mathbb{Z}/2$. The circles are meant to indicate that the generator is twice what one might expect (for example,~$w$ is `twice a square root of~$U$'). For more details, we refer to \cite{Gre18}.}
\label{figure:kr}
\end{figure}

Passing to~$C_2$--fixed points, the splitting (\ref{eq:splitcnr}) shows that there is an isomorphism
\[ \pi_*(C_n(O))\cong \bigoplus _{k=1}^n \bigoplus_{n\choose k} \pi^{C_2}_{*-k\cdot \sigma}(k\R).\]

It is most convenient to describe this graded group as a graded module over~$\pi_\ast(ko)$. Let~$I$ denote the augmentation ideal of~$\pi_\ast(ko)$, and~$J$ the ideal generated by~$\eta^2,x$ and~$y$. Define
\begin{equation} \label{eq:ak}
A(k)=\begin{cases} (\bigoplus_{i=0}^{j-1}\Z/2[4i])\oplus \pi_\ast(ko)[4j] & k=4j \\
											(\bigoplus_{i=0}^{j-1}\Z/2[4i])\oplus I[4j-1]  & k=4j+1 \\
											(\bigoplus_{i=0}^{j-1}\Z/2[4i])\oplus J[4j-2]  & k=4j+2 \\
											(\bigoplus_{i=0}^{j}\Z/2[4i])\oplus \pi_\ast(ko)[4j+5]  & k=4j+3\,,
								\end{cases}
\end{equation}
where the copies of~$\Z/2$ are acted on by~$\pi_\ast(ko)$ through the augmentation~$\pi_\ast(ko)\to \Z$, i.e.,~$\eta, x$ and~$y$ act trivially.
Then there is an isomorphism of~$\pi_\ast(ko)$--modules
\[ \pi_*(C_n(O))\cong \bigoplus_{k=1}^n \bigoplus_{n\choose k} A(k). \]
It is possible to write down an explicit formula for~$\pi_k(C_n(O))$ as a function of~$k$ and~$n$, but the result does not seem very enlightening. Up to~$k=7$ the groups are listed in Figure \ref{figure:lowhomotopy}. The computation of~$\pi_0(C_n(O))\cong (\Z/2)^{2^n-1}$ is originally due to Rojo \cite{Ro14}, while the fact that~$\pi_1(C_n(O))\cong (\Z/2)^n$ follows from the more general result of G\'omez, Pettet and Souto \cite{GPS12} that~$\pi_1(C_n(G))$ (with~$C_n(G)$ given its natural basepoint) is isomorphic to~$\pi_1(G)^{n}$ for every compact Lie group~$G$.
\begin{figure}[h]
{\def\arraystretch{1.5}\tabcolsep=8pt
\begin{tabular}{l|l||l|l}
$k$ &~$\pi_k(C_n(O))$	&~$k$ &~$\pi_k(C_n(O))$                       \\ \hline\hline
$0$ &~$(\Z/2)^{2^n-1}$	&~$4$ &~$\Z^{n \choose 4}\oplus (\Z/2)^{\sum_{i=5}^{n} {n \choose i}}$ \\ \hline
$1$ &~$(\Z/2)^n$		&~$5$ &~$\Z^{n \choose 3} \oplus (\Z/2)^{{n \choose 4}+ {n \choose 5}}$   \\ \hline
$2$ &~$\Z^{n\choose 2}$	&~$6$ &~$\Z^{{n\choose 2}+{n\choose 6}}\oplus (\Z/2)^{{n\choose 3}+{n\choose 4}}$   \\ \hline
$3$ &~$\Z^n$			&~$7$ &~$\Z^{n+{n\choose 5}}\oplus (\Z/2)^{{n\choose 2}+{n\choose 3}}$\\ \hline
\end{tabular}
}
%&~$\Z^{n+{n\choose 5}}\oplus \Z/2^{{n\choose 2}+{n\choose 3}}$
\caption{The first eight homotopy groups of~$C_n(O)$}
\label{figure:lowhomotopy}
\end{figure}

\begin{Remark}
With little effort, our computation of the homotopy groups of~$C_n(O)$ also leads to a computation of the homotopy groups of the space of commuting~$n$--tuples in the stable spinor group~$Spin$. The main observation is that~$C_n(Spin)$ is a group-like H--space and that the natural map~$C_n(Spin)\to C_n(SO)$ is a covering over every connected component of its image. This can be used to show that there is an equivalence
\[
C_n(Spin)\simeq \pi_0(C_n(Spin))\times \widetilde{C_{n,\mathds{1}}}(O)\, ,
\]
where~$\widetilde{C_{n,\mathds{1}}}(O)$ denotes the universal cover of~$C_{n,\mathds{1}}(O)$, the component of~$C_n(O)$ containing the $n$--tuple of identity matrices. One can easily compute the group
\[
\pi_0(C_n(Spin))\cong (\Z/2)^{2^n-n-1-{n \choose 2}}\,,
\]
by showing that it is isomorphic to the kernel of the Stiefel-Whitney class~$w_2\colon\thinspace \pi_0(C_n(SO))\to H^2(\Z^n;\Z/2)$ (this uses again the fact that~$\pi_0(C_n(Spin))$ is a group).
\end{Remark}

\subsubsection{The homotopy groups of $\Rep_n(O)$} \label{sec:rno} Again it is enough to consider the subring of~$\pi^{C_2}_{{\Star}}(H\underline{\Z})$ in those degrees~$k+l\cdot \sigma$ where~$k\geq l$. It is generated by two elements:
\begin{itemize}
	\item The Euler class~$a$ in degree~$-\sigma$ obtained via the composite~$S^0\to H\underline{\Z}\wedge S^0\to  H\underline{\Z}\wedge S^{\sigma}$.
	\item	The orientation class~$u$ in degree~$2-2\cdot \sigma$.
\end{itemize}
The only relation is~$2a=0$. See Figure \ref{figure:hz} for an illustration of this ring. The map~$\pi_{{\Star}}^{C_2}(k\R)\to \pi_{{\Star}}^{C_2}(H\underline{\Z})$ sends~$a$ to~$a$,~$\overline{v}$ to~$0$,~$w$ to~$2u$ and~$U$ to~$u^2$. In particular, the element~$u$ is not in the image, and as a consequence the natural map~$\pi_*(C_n(O))\to \pi_*(\Rep_n(O))$ is not surjective for~$n\geq 2$, unlike the unitary case. Reading off the columns in Figure \ref{figure:hz}, one obtains
\[ \pi_{2i}(\Rep_n(O))\cong \Z^{n\choose 2i}\oplus (\Z/2)^{\sum_{k=2i+1}^n {n\choose k}}\,, \hspace{0.8cm}  \pi_{2i-1}(\Rep_n(O))=0 \]
for all~$i\geq 1$, and~$\pi_0(\Rep_n(O))\cong \pi_0(C_n(O))$.

\begin{Remark}
It is also possible to compute the groups~$\pi_{i}(\Rep_n(O))$ for~$i\geq 1$ using the methods of Adem, Cohen and G\'omez \cite{ACG10}. If~$G$ is a compact, connected Lie group such that~$\Rep_n(G)$ is path-connected for every~$n$, then by \cite[Theorem 6.1]{ACG10} there is a homeomorphism~$\Rep_n(G)\cong T^n/W$, where~$T\leqslant G$ is a maximal torus and~$W$ is the Weyl group acting diagonally on~$T^n$. This theorem continues to hold even if~$\Rep_n(G)$ is not path-connected, so long as we restrict ourselves to the path-component of the trivial homomorphism, which we denote by~$\Rep_{n,\mathds{1}}(G)$. We can view~$\Rep_{n,\mathds{1}}(O)$ as the colimit of~$\Rep_{n,\mathds{1}}(SO(2m+1))$ as~$m$ goes to infinity. A maximal torus for~$SO(2m+1)$ is homeomorphic to~$(S^1)^{\times m}$ with Weyl group the wreath product~$C_2 \wr \Sigma_m$ acting on~$(S^1)^{\times m}$ by permutation and complex conjugation of the factors. It follows that~$\Rep_{n,\mathds{1}}(O)\cong Sp^\infty((S^\sigma)^{\times n}/C_2)$. This space is homeomorphic to~$\Rep_n(Sp)$, where~$Sp$ is the stable quaternionic unitary group, i.e., the colimit over~$m$ of~$Sp(m)$, the automorphism group of~$\HH^{m}$ preserving the standard hermitian form. The homotopy groups of~$\Rep_n(Sp)$ were determined in \cite[Theorem 6.8]{ACG10} and indeed, their result agrees with ours for~$\pi_i(\Rep_{n,\mathds{1}}(O))$.
\end{Remark}

\begin{figure}[h]
\begin{tikzpicture}[scale =0.8]
\small

\clip (-4, -4) rectangle (6, 6); % Ausschneiden
\draw[step=0.5, gray, very thin] (-4,-4) grid (6, 6);

\node at (1,1)  [shape = rectangle, draw]{};
\draw (1,1) node[anchor=east]{$1\,$};

\node at (2,0) [shape=rectangle, draw] {};
\draw (2,0) node[anchor=east]{$u\,$};

\draw[red] (1,1) -- (1,-4);

\foreach \x in {1, 2, 3,4,5,6,7,8,9,10}
\draw (1,1-\x/2) node[anchor=east] {$a^{\x}$};
\foreach \x in {1,2,3,4,5,6,7,8,9,10}
\node at (1,1-\x/2) [fill=red, inner sep=1pt, shape=circle, draw] {};

\node at (3,-1)  [shape = rectangle, draw]{};
\draw (3,-1) node[anchor=east]{$u^2\,$};

\node at (4,-2) [shape=rectangle, draw] {};
\draw (4,-2) node[anchor=east]{$u^3\,$};

\draw[red] (3,-1) -- (3,-4);

\draw[blue, ->, thick](2.6,1)--(5.8, 1); % horizontale Achse / Pfeil
\draw (3.9, 1) node[anchor=north]{$\Z\cdot 1$};

\foreach \x in {1,2,3,4,5,6,7,8}
\node at (2,-\x/2) [fill=red, inner sep=1pt, shape=circle, draw] {};

\draw[red] (2,0) -- (2,-4);

\foreach \x in {1,2,3,4,5,6}
\node at (3,-1-\x/2) [fill=red, inner sep=1pt, shape=circle, draw] {};

\draw[red] (4,-2) -- (4,-4);

\foreach \x in {1,2,3,4}
\node at (4,-2-\x/2) [fill=red, inner sep=1pt, shape=circle, draw] {};

\node at (5,-3)  [shape = rectangle, draw]{};
\draw (5,-3) node[anchor=east]{$u^4\,$};

\node at (6,-4) [shape= rectangle, draw] {};

\draw[red] (5,-3) -- (5,-4);
\foreach \x in {1,2}
\node at (5,-3-\x/2) [fill=red, inner sep=1pt, shape=circle, draw] {};

% The following part is not relevant to us

\draw[blue, ->, thick](1,2.6)--(1, 5.8); % vertikale Achse / Pfeil
\draw (1,3.9) node[anchor=east]{$\Z\cdot \sigma$};

\node at (-1,3)  [shape = circle, draw]{};
\draw[red](-1.5, 3.5) -- (-1.5, 6);
\foreach \x in {0, 1,2,3,4,5}
\node at (-1.5,3.5+\x/2) [fill=red, inner sep=1pt, shape=circle, draw] {};

\node at (0,2)  [shape = circle, draw]{};
\draw[red](-0.5, 2.5) -- (-0.5, 6);
\foreach \x in {0, 1,2,3,4,5,6,7}
\node at (-0.5,2.5+\x/2) [fill=red, inner sep=1pt, shape=circle, draw] {};

\node at (-3,5)  [shape = circle, draw]{};
\draw[red](-3.5, 5.5) -- (-3.5, 6);
\foreach \x in {0,1}
\node at (-3.5,5.5+\x/2) [fill=red, inner sep=1pt, shape=circle, draw] {};

\node at (-2,4)  [shape = circle, draw]{};
\draw[red](-2.5, 4.5) -- (-2.5, 6);
\foreach \x in {0,1,2,3}
\node at (-2.5,4.5+\x/2) [fill=red, inner sep=1pt, shape=circle, draw]
{};

\end{tikzpicture}
\caption{The ring~$\pi_{\Star}^{C_2}(H\underline{\Z})$. The part which is relevant for our computation lies in the fourth quadrant. Black circles and squares represent copies of~$\Z$; red dots represent copies of~$\Z/2$.}
\label{figure:hz}
\end{figure}
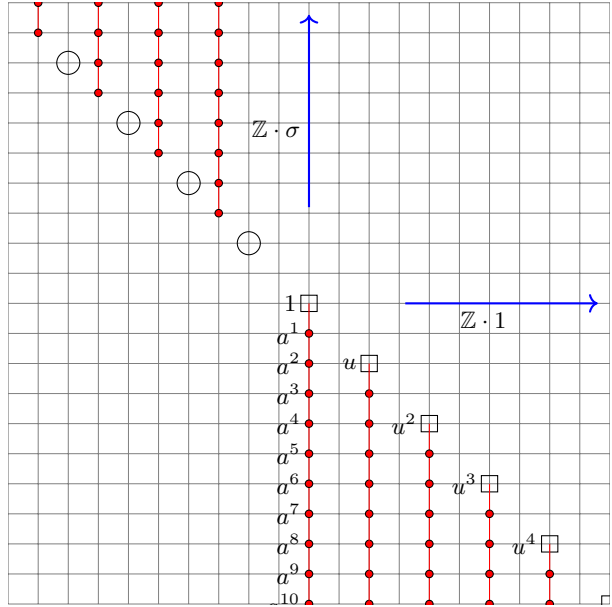

\subsubsection{The homotopy ring of $B_{\textnormal{com}}O$} \label{sec:bcomo}
We now come to the classifying space for commutative orthogonal~K-theory. From the splitting (\ref{eq:splitbcom}) we see that
\[ \pi_*(B_{\textnormal{com}}O)\cong \pi_*((k\R\wedge \CP^{\infty})^{C_2})\cong \bigoplus_{k>0}\pi_*((k\R\wedge S^{k\cdot (1+\sigma)})^{C_2})\cong \bigoplus_{k>0} \pi^{C_2}_{*-k\cdot(1+ \sigma)} (k\R). \]
Thus, there is an isomorphism of~$\pi_\ast(ko)$--modules
\[ \pi_*(B_{\textnormal{com}}O)\cong \bigoplus_{k>0} A(k)[k],\]
where~$A(k)[k]$ denotes the~$k$--fold shift of the graded~$\pi_\ast(ko)$--module~$A(k)$ from Definition (\ref{eq:ak}).

The sphere~$S^{\sigma}$ is an abelian~$C_2$--group, and hence so is its classifying space~$BS^{\sigma}$, which is~$C_2$--homotopy equivalent to $\CP^{\infty}$. It follows that~$k\R\wedge \CP^{\infty}_+$ is an augmented $k\R$-algebra. As a consequence,~$\pi_{\Star}^{C_2}(k\R\wedge \CP^\infty_+)$ is an augmented $\pi_{\Star}^{C_2}(k\R)$--algebra, whose augmentation ideal is naturally isomorphic to~$\pi_{\Star}^{C_2}(k\R\wedge \CP^{\infty})$. In particular,~$\pi_{\Star}^{C_2}(k\R\wedge \CP^{\infty})$ as well as the integer graded part~$\pi_*^{C_2}(k\R\wedge \CP^{\infty})\cong \pi_*(B_{\textnormal{com}}O)$ are non-unital rings, similar to~$\pi_*(ku\wedge \CP^{\infty})$ discussed in Section \ref{sec:bcomu}. The ring structure is best described by first considering the unital~$RO(C_2)$--graded ring~$\pi_{\Star}^{C_2}(k\R\wedge \CP^{\infty}_+)$ and reducing to the augmentation ideal and the~$\Z$--graded part in the end. 

The spectrum~$k\R$ is Real oriented, and the Real orientation chosen in Section \ref{sec:krmodule} restricts to the complex orientation of~$ku$ specified in Section \ref{sec:bcomu}. It follows that~$\pi_{\Star}^{C_2}(k\R\wedge \CP^{\infty}_+)$ is free over~$\pi_{\Star}^{C_2}(k\R)$ on generators~$\overline{y}_i\in \pi^{C_2}_{(1+ \sigma)\cdot i}(k\R\wedge \CP^{\infty}_+)$ for~$i\geq 0$, which restrict to the generators~$y_i\in \pi_{2i}(ku\wedge \CP^{\infty}_+)$ from Section \ref{sec:bcomu}. Moreover, the restriction maps
\[
\pi_{k+l\cdot \sigma}^{C_2}(k\R\wedge \CP^{\infty}_+) \to \pi_{k+l}(ku\wedge \CP^{\infty}_+)
\]
induce an isomorphism of~$\Z$--graded rings
\[
\pi_{(1+ \sigma)*}^{C_2}(k\R\wedge \CP^{\infty}_+) \cong \pi_{2*}(ku\wedge \CP^{\infty}_+)\,.
\]
Indeed, the corresponding statement holds for~$k\R$, where all elements of~$\pi^{C_2}_{(1+ \sigma)*}(k\R)$ are multiples of powers of the Bott element~$\overline{v}$. It then follows for~$k\R\wedge \CP^{\infty}_+$, since the generators~$\overline{y}_i$ lie in degrees~$(1+ \sigma)\cdot i$.

As a consequence, the~$\overline{y}_i$ satisfy the same relations as the~$y_i$ and we have
\begin{equation} \label{eq:relations1}
\overline{y}_k\overline{y}_l=\sum_{i=\max(k,l)}^{k+l}\left( \frac{i!}{(i-k)!(i-l)!(k+l-i)!} \overline{v}^{k+l-i}\cdot \overline{y}_i\right)
\end{equation}
for all~$k,l\geq 0$ (see (\ref{eq:ykyl})). Let this set of relations be denoted by~$\overline{R}$. We obtain an induced map
\begin{equation} \label{eq:relations2}
\varphi\co \pi_{\Star}^{C_2}(k\R)[\{\overline{y}_n\}]/\overline{R}\to \pi_{\Star}^{C_2}(k\R\wedge \CP^{\infty}_+)\, ,
\end{equation}
at least after we note that the right hand side is a commutative ring, or in other words that the coefficients~$\pi_{\Star}^{C_2}(k\R)$ commute with the~$\overline{y}_i$. To see this, we note that the~$C_2$--ring spectrum~$k\R\wedge \CP^{\infty}_+$ is Real oriented, since~$k\R$ is. So we can apply \cite[Lemmas 2.12 and 2.17]{HK01} to obtain that
\[
\alpha \overline{y}_i=(-1)^{i\cdot (k+l)}\overline{y}_i \alpha
\]
for~$\alpha\in \pi_{k+l\cdot \sigma}^{C_2} (k\R)$. Since all such~$\alpha$ with~$k+l$ an odd number are~$2$--torsion (see Figure \ref{figure:kr}), this shows that~$\pi_{\Star}^{C_2}(k\R\wedge \CP^{\infty}_+)$ is indeed commutative.

\begin{Prop} The map \eqref{eq:relations2} is an isomorphism.
\end{Prop}
\begin{proof} Since~$\pi_{\Star}^{C_2}(k\R\wedge \CP^{\infty}_+)$ is free as a~$\pi_{\Star}^{C_2}(k\R)$--module on the~$\overline{y}_i$, we can define a~$\pi_{\Star}^{C_2}(k\R)$--module map~$\psi$ in the other direction by sending~$\overline{y}_i$ to~$\overline{y}_i$. It is clear that~$\varphi\circ \psi$ equals the identity. Furthermore,~$(\psi \circ \varphi)(\overline{y}_i)=\overline{y}_i$, so it suffices to show that the domain of \eqref{eq:relations2} is generated by the~$\overline{y}_i$ as a~$\pi_{\Star}^{C_2}(k\R)$--module. But this is clear, since by (\ref{eq:relations1}) every product~$\overline{y}_k\overline{y}_l$ can be replaced by a~$\pi_{\Star}^{C_2}(k\R)$--linear combination of the~$\overline{y}_i$. This finishes the proof.
\end{proof}
For the integer graded part of the augmentation ideal, we have an isomorphism
\[\pi_*^{C_2}(k\R\wedge \CP^{\infty})\cong \bigoplus_{k>0} \pi_{*-k\cdot(1+ \sigma)}^{C_2}(k\R)\cdot \overline{y}_k. \]
Given~$\alpha \in \pi_{*-k\cdot(1+ \sigma)}^{C_2}(k\R)$ and~$\beta\in \pi_{*-l\cdot(1+ \sigma)}^{C_2}(k\R)$, the product of~$\alpha \cdot \overline{y}_k$ and~$\beta \cdot \overline{y}_l$ computes as
\[ (\alpha\cdot \overline{y}_k)\cdot (\beta\cdot \overline{y}_l)=(\alpha \beta)\cdot (\overline{y}_k \overline{y}_l)=(\alpha \beta)\cdot \sum_{i= \max(k,l)}^{k+l}\left(\frac{i!}{(i-k)!(i-l)!(k+l-i)!} \overline{v}^{k+l-i}\cdot \overline{y}_i\right).\]
Hence, if we define 
\[ m_i^{k,l}:\pi_{*-k\cdot(1+ \sigma)}^{C_2}(k\R)\otimes \pi_{*-l\cdot(1+ \sigma)}^{C_2}(k\R)\to \pi_{*-i\cdot(1+ \sigma)}^{C_2}(k\R) \]
by the product (formed inside~$\pi_{\Star}^{C_2}(k\R)$)
\[ m_i^{k,l}(\alpha\otimes \beta)=\alpha \beta \overline{v}^{k+l-i}, \]
we obtain:
\begin{Cor} The multiplication~$m$ on 
\[ \pi_*(B_{\textnormal{com}}O)\cong \pi_*^{C_2}(k\R\wedge \CP^{\infty})\cong \bigoplus_{k>0}\pi^{C_2}_{*-k\cdot(1+ \sigma)}(k\R)\cdot \overline{y}_k \]
is given by the formula
\[ m((\alpha\cdot \overline{y}_k)\otimes (\beta\cdot \overline{y}_l))= \sum_{i= \max(k,l)}^{k+l} \left( \frac{i!}{(i-k)!(i-l)!(k+l-i)!} m_i^{k,l}(\alpha\otimes \beta) \cdot \overline{y}_i\right).\]
\end{Cor}
Note that if~$\alpha\beta$ is divisible by~$a^3$, all multiplications~$m_i^{k,l}(\alpha\otimes \beta)$ are trivial except for~$m_{k+l}^{k,l}$ and the product~$m((\alpha\cdot \overline{y}_k)\otimes (\beta\cdot \overline{y}_l))$ is just~$\frac{(k+l)!}{k!l!}$ times the usual product~$\alpha\beta \cdot \overline{y}_{k+l}$.
\begin{Example} The product of the two classes~$aU\cdot \overline{y}_5$ and~$U\overline{v}\cdot \overline{y}_3$ is given by
\[ m((aU\cdot \overline{y}_5)\otimes (U\overline{v}\cdot \overline{y}_3))= 10 aU^2\overline{v}^4\cdot \overline{y}_5 + 60 aU^2\overline{v}^3 \cdot \overline{y}_6 + 105  aU^2\overline{v}^2 \cdot \overline{y}_7 + 56 aU^2\overline{v} \cdot \overline{y}_8. \]
As~$a$ is~$2$--torsion, this simplifies to
\[ m((aU\cdot \overline{y}_5)\otimes (U\overline{v}\cdot \overline{y}_3))= aU^2\overline{v}^2\cdot \overline{y}_7. \]
\end{Example}

\subsection{Description as FI-modules and representation stability} \label{sec:fi}
Ramras and Stafa show in \cite[Theorem 7.7]{RS18} that the rational homology groups of the spaces~$C_n(U(m))$,~$\Rep_n(U(m))$ as well as~$C_{n,\mathds{1}}(O(m))$,~$\Rep_{n,\mathds{1}}(O(m))$ -- for fixed~$m$  and varying~$n$ -- are representation stable in the sense of \cite{CEF15}, with stability range independent of~$m$. The homology groups also satisfy homological stability in~$m$, by \cite[Theorem 1.2]{RS18}, hence it follows that the rational homology groups of the spaces~$C_n(U)$,~$\Rep_n(U)$,~$C_n(O)$ and~$\Rep_n(O)$ are also representation stable in~$n$. One might wonder whether the (integral) homotopy groups~$\pi_k(-)$ of these spaces are also representation stable. The answer is `yes' in the unitary case, whereas in the orthogonal case it turns out to depend on~$k$.

To phrase representation stability we consider the functor
\[ \underline{n}\mapsto C_n^{\R} \ ; \ (\alpha\co \underline{n}\hookrightarrow \underline{m})\mapsto \alpha_*(A_1,\hdots,A_n)= (A_{\alpha^{-1}(1)},\hdots,A_{\alpha^{-1}(m)})  \]
from the category of finite sets and injections FI to the category of~$C_2$--spaces, and likewise for~$\Rep_n^{\R}$. Here,~$\underline{n}=\{1,\dots,n\}$ and~$A_{\alpha^{-1}(i)}=\id$ if~$i$ does not lie in the image of~$\alpha$. 
Given a~$\Sigma_n$--module~$A$, we write~$\Ind_{\Sigma_n}^{\FI}A$ for the left induced FI-module~$\Z[\FI(\underline{n},-)]\otimes_{\Sigma_n} A$, i.e., the left adjoint to evaluation on~$\underline{n}$ applied to~$A$. Then the following is perhaps the most compact way to state the computation for the homotopy groups of the spaces~$C_n(U),\Rep_n(U),C_n(O)$ and~$\Rep_n(O)$:
\begin{Prop} \label{prop:fi} There are isomorphisms of FI-modules
\begin{align*} \pi_k(C_-(U))\cong & \bigoplus_{n>0} \Ind_{\Sigma_n}^{\FI}(\pi_{k-n}(ku)) \\
\pi_k(\Rep_-(U))\cong & \bigoplus_{n>0} \Ind_{\Sigma_n}^{\FI}(\pi_{k-n}(H\Z)) \\
\pi_k(C_-(O))\cong & \bigoplus_{n>0} \Ind_{\Sigma_n}^{\FI}(\pi_{k-n\cdot \sigma}^{C_2}(k\R)) \\
\pi_k(\Rep_-(O))\cong & \bigoplus_{n>0} \Ind_{\Sigma_n}^{\FI}(\pi_{k-n\cdot \sigma}^{C_2}(H\underline{\Z}))\, ,
\end{align*}
where~$\Sigma_n$ acts on each of the respective homotopy groups by multiplication with the sign.
\end{Prop}
\begin{proof}
As follows from the proof of Theorem \ref{thm:main}, under the equivalences
\[ C_n^{\R}\simeq \Omega^{\infty} ((k\R\wedge S^{\sigma})^{\times n}) \ \text{and} \ \Rep_n^{\R}\simeq \Omega^{\infty}((H\underline{\Z}\wedge S^{\sigma})^{\times n}) \]
the maps~$\alpha_*$ are induced by the~$C_2$--maps~$f_{\alpha}\co (S^{\sigma})^{\times n}\to (S^{\sigma})^{\times m}$ defined via~$f_{\alpha}(x_1,\hdots,x_n)=(x_{\alpha^{-1}(1)},\hdots, x_{\alpha^{-1}(m)})$, where~$x_{\alpha^{-1}(i)}=*$ if~$i$ does not lie in the image of~$\alpha$.

Taking this FI-module structure into account, the stable splitting of the products~$(S^{\sigma})^{\times n}$ can be phrased by saying that there is an equivalence of functors~$\FI\to \Ho(Sp_{C_2})$
\[ \Sigma^{\infty} (S^{\sigma})^{\times -}\simeq \bigvee_{n>0}\Sigma^{\infty} (\FI(n,-)_+\wedge_{\Sigma_n} S^{n\cdot \sigma}),\]
with~$\Sigma_n$ acting on~$S^{n\cdot \sigma}$ by permuting the coordinates. The result then follows by smashing with~$k\R$ or~$H\underline{\Z}$ and taking the respective homotopy groups. Non-equivariantly it is clear that the resulting~$\Sigma_n$--action is through the sign representation. For the equivariant homotopy groups we again make use of \cite[Lemma 2.17]{HK01}, which implies that after smashing with~$k\R$ or~$H\underline{\Z}$, the symmetry isomorphism~$\tau\co S^{\sigma}\wedge S^{\sigma}\xrightarrow{\cong} S^{\sigma}\wedge S^{\sigma}$ becomes homotopic to~$-1$.
\end{proof}
Recall that an FI-module $M$ is called `representation stable', or `centrally stable', if there exists an $m\in \N$ such that the natural map 
\[ L^{\FI}_{\leq m}i^*_{\leq m}(M)\to M \]
is an isomorphism (cf. \cite[Theorem C]{CEFN14}). Here, $i^*_{\leq m}(M)$ denotes the restriction of $M$ to the full subcategory $\FI_{\leq m}$ of FI spanned by all $\underline{l}$ with $l\leq m$, and $L^{\FI}_{\leq m}$ denotes the left Kan extension back up to an~FI-module. If this map is an isomorphism for a given $m$, we say that `M is representation stable in degrees~$n\geq m$'. This condition implies by definition that $M$ can be reconstructed from its values~$M(\underline{l})$ for $l\leq m$ using the FI-module structure. In this sense, $M$ stabilizes at $m$.

Now, for any $\Sigma_n$-module $A$, one has 
\[ L_{\leq m}^{\FI}i^*_{\leq m}(\Ind_{\Sigma_n}^{\FI}(A))\xrightarrow{\cong} \Ind_{\Sigma_n}^{\FI}(A) \]
for $m\geq n$ and $L_{\leq m}^{\FI}i^*_{\leq m}(\Ind_{\Sigma_n}^{\FI}(A))=0$ for $m<n$. (To see the former, note that if $m\geq n$, then $\Ind_{\Sigma_n}^{\FI}$ is the left Kan extension along the composite $B\Sigma_n\hookrightarrow \FI_{\leq m}\hookrightarrow \FI$, and hence, by transitivity, $\Ind_{\Sigma_n}^{\FI}(A)$ is in the image of $L_{\leq m}^{\FI}$. Thus $L_{\leq m}^{\FI}i^*_{\leq m}(\Ind_{\Sigma_n}^{\FI}(A))\to \Ind_{\Sigma_n}^{\FI}(A)$ is an isomorphism, since left Kan extensions along full embeddings are fully-faithful.) Hence it follows that~$\pi_k(-)$ is representation stable if and only if almost all of the~$\Sigma_n$--modules on the right hand side of the respective equivalence in Proposition \ref{prop:fi} are trivial, in which case the top non-trivial homotopy group yields the degree of stabilization.

\begin{Cor} \leavevmode \begin{itemize} \item The FI-modules~$\pi_k(C_-(U))$ and~$\pi_k(\Rep_-(U))$ are representation stable in degrees~$n\geq k$.
\item The FI-module~$\pi_k(C_-(O))$ is representation stable if and only if~$k\not\equiv 0\, (\text{mod }4)$. In these cases, it is stable in degrees~$n\geq k$ if~$k\equiv 1,2 \, (\text{mod }4)$ and it is stable in degrees~$n\geq (k-2)$ if~$k\equiv 3\, (\text{mod }4)$.
\item The FI-module~$\pi_k(\Rep_-(O))$ is representation stable if and only if~$k$ is odd, in which case it is the~$0$--module.
\end{itemize}
\end{Cor}

\end{document}